\documentclass[a4paper,10pt,leqno]{amsart}
\title{Assembly Maps}

\makeindex

\author{Wolfgang L\"uck}
        \email{wolfgang.lueck@him.uni-bonn.de}
          \urladdr{http://www.him.uni-bonn.de/lueck}

        \address{Mathematisches Institut der Universit\"at Bonn\\
                Endenicher Allee 60\\
                53115 Bonn, Germany}

         \date{January, 2019}

\usepackage{color}
\usepackage{pdfsync}
\usepackage{calc}
\usepackage{enumerate,amssymb}
\usepackage[arrow,curve,matrix,tips,2cell]{xy}
  \SelectTips{eu}{10} \UseTips
  \UseAllTwocells

\DeclareMathAlphabet\EuR{U}{eur}{m}{n}
\SetMathAlphabet\EuR{bold}{U}{eur}{b}{n}

\makeindex             

\theoremstyle{plain}
\newtheorem{theorem}{Theorem}[section]
\newtheorem{lemma}[theorem]{Lemma}

\newtheorem{conjecture}[theorem]{Conjecture}
\newtheorem{problem}[theorem]{Problem}

\theoremstyle{definition}

\newtheorem{definition}[theorem]{Definition}

\newtheorem{remark}[theorem]{Remark}

{\catcode`@=11\global\let\c@equation=\c@theorem}

\newcommand{\comsquare}[8]                   
{\begin{CD}
#1 @>#2>> #3\\
@V{#4}VV @V{#5}VV\\
#6 @>#7>> #8
\end{CD}
}

\newcommand{\xycomsquare}[8]                   
{\xymatrix
{#1 \ar[r]^{#2} \ar[d]^{#4} &
#3 \ar[d]^{#5}  \\
#6\ar[r]^{#7} &
#8
}
}

\newcommand{\xycomsquareminus}[8]                      
{\xymatrix{#1 \ar[r]^-{#2} \ar[d]^-{#4} &
#3 \ar[d]^-{#5}  \\
#6\ar[r]^-{#7} &
#8
}
}

\newcommand{\calall}{{\mathcal A} {\mathcal L}{\mathcal L}}
\newcommand{\calbc}{{\mathcal B}{\mathcal C}}
\newcommand{\calfcyc}{{\mathcal F}{\mathcal C}{\mathcal Y}}
\newcommand{\calcyc}{{\mathcal C}{\mathcal Y}}
\newcommand{\calfj}{{\mathcal F}\!{\mathcal J}}
\newcommand{\calfin}{{\mathcal F}{\mathcal I}{\mathcal N}}
\newcommand{\calvcyc}{{\mathcal V}{\mathcal C}{\mathcal Y}}
\newcommand{\calT}{{\mathcal T}}
\newcommand{\caltr}{{\mathcal T}{\mathcal R}}

\newcommand{\calf}{\mathcal{F}}
\newcommand{\calg}{\mathcal{G}}
\newcommand{\calh}{\mathcal{H}}
\newcommand{\caln}{{\mathcal N}}

\newcommand{\cals}{\mathcal{S}}

\newcommand{\IC}{{\mathbb C}}

\newcommand{\IQ}{{\mathbb Q}}
\newcommand{\IR}{{\mathbb R}}

\newcommand{\IZ}{{\mathbb Z}}

\newcommand{\bfA}{{\mathbf A}}
\newcommand{\bfB}{{\mathbf B}}

\newcommand{\bfE}{{\mathbf E}}

\newcommand{\bff}{{\mathbf f}}
\newcommand{\bfF}{{\mathbf F}}

\newcommand{\bfK}{{\mathbf K}}

\newcommand{\bfL}{{\mathbf L}}

\newcommand{\bfR}{{\mathbf R}}

\newcommand{\bfT}{{\mathbf T}}
\newcommand{\bfTC}{{\mathbf {TC}}}
\newcommand{\bfTHH}{{\mathbf {THH}}}
\newcommand{\bft}{{\mathbf t}}

\newcommand{\curs}{\EuR}

\newcommand{\GCW}{G\text{-}\curs{CW}}
\newcommand{\GCWpairs}{G\text{-}\curs{CW}^2}

\newcommand{\Groupoids}{\curs{Groupoids}}

\newcommand{\Or}{\curs{Or}}
\newcommand{\OrG}{\curs{Or}(G)}

\newcommand{\Spaces}{\curs{Spaces}}
\newcommand{\Spectra}{\curs{Spectra}}
\newcommand{\Sets}{\curs{Sets}}

\newcommand{\asmb}{\operatorname{asmb}}

\newcommand{\CAT}{\operatorname{CAT}}

\newcommand{\colim}{\operatorname{colim}}

\newcommand{\cone}{\operatorname{cone}}
\newcommand{\con}{\operatorname{con}}

\newcommand{\hocolim}{\operatorname{hocolim}}

\newcommand{\id}{\operatorname{id}}

\newcommand{\ind}{\operatorname{ind}}

\newcommand{\map}{\operatorname{map}}

\newcommand{\mor}{\operatorname{mor}}

\newcommand{\pr}{\operatorname{pr}}

\newcommand{\Spin}{\operatorname{Spin}}

\newcommand{\topo}{\operatorname{top}}

\newcommand{\tr}{\operatorname{tr}}

\newcommand{\Wh}{\operatorname{Wh}}

\newcommand{\pt}{\{\bullet\}}
\newcommand{\NK}{N\!K}

\newcommand{\EGF}[2]{E_{#2}(#1)}                   
\newcommand{\eub}[1]{\underline{E}#1}              
\newcommand{\edub}[1]{\underline{\underline{E}}#1} 
\newcommand{\OrGF}[2]{\Or_{#2}(#1)}                

\newcommand{\higherlim}[3]{{\setbox1=\hbox{\rm lim}
        \setbox2=\hbox to \wd1{\leftarrowfill} \ht2=0pt \dp2=-1pt
        \mathop{\vtop{\baselineskip=5pt\box1\box2}}
        _{#1}}^{#2}#3}

\newcommand{\version}[1]                       
{\begin{center} last edited on #1\\
last compiled on \today\\
name of texfile: \jobname
\end{center}
}

\newcounter{commentcounter}

\begin{document}

\typeout{---------------------------- assembly.tex ----------------------------}


\typeout{------------------------------------ Abstract
  ----------------------------------------}

\begin{abstract}
  We introduce and analyze the concept of an assembly map from the original homotopy
  theoretic point of view.  We give also interpretations in terms of surgery theory,
  controlled topology and index theory.  The motivation is that prominent conjectures of
  Farrell-Jones and Baum-Connes about $K$- and $L$-theory of group rings and group
  $C^*$-algebras predict that certain assembly maps are weak homotopy equivalences.
\end{abstract}

\maketitle


\typeout{------------------------------- Section 0: Introduction
  --------------------------------}

\setcounter{section}{-1}
\section{Introduction}


\subsection{The homotopy theoretic description of assembly maps}
\label{subsec:The_homotopy_theoretic_description_of_assembly_maps_intro}

The quickest and probably for a homotopy theorist most convenient approach to assembly
maps is via homotopy colimits as explained in
Subsection~\ref{subsec:The_assembly_map_in_terms_of_homotopy_colimits}.  Let $\calf$ be a
family of subgroups of $G$, i.e., a collection of subgroups closed under conjugation and
passing to subgroups.  Let $\OrG$ be the orbit category%
\index{orbit category}
and $\OrGF{G}{\calf}$ be the full
subcategory consisting of objects $G/H$ satisfying $H \in \calf$. Consider a covariant functor
$\bfE^G \colon \OrG \to \Spectra$ to the category of spectra. We get from the
inclusion $\OrGF{G}{\calf} \to \OrG$ and the fact that $G/G$ is a terminal object in
$\OrG$ a map
\begin{equation}
\hocolim_{\OrGF{G}{\calf}} \bfE^G|_{\OrGF{G}{\calf}} \to \hocolim_{\OrG} \bfE^G =
  \bfE^G(G/G).
\label{asmb_map_hocolim_intro}
\end{equation}
It is called \emph{assembly map} since we are trying to assemble the values of $\bfE^G$ on homogeneous
spaces $G/H$ for $H \in \calf$ to get $\bfE(G/G)$.

On homotopy groups this assembly map can also be described as the map
\begin{equation}
H_n^G(\pr;\bfE^G) \colon H^G_n(\EGF{G}{\calf};\bfE^G) \to H_n^G(G/G;\bfE^G) = \pi_n(\bfE^G(G/G))
\label{asmb_homologically_intro}
\end{equation}
induced by the projection $\pr \colon \EGF{G}{\calf} \to G/G$ of the classifying $G$-space $\EGF{G}{\calf}$
for the family $\calf$, see Section~\ref{sec:Classifying_spaces_for_families_of_subgroups}, to $G/G$,
where $H_n^G(-;\bfE^G)$ is the $G$-homology theory in the sense of Definition~\ref{def:G-homology_theory}
associated to $\bfE^G$, see Lemma~\ref{lem:extending_from_Or(G)_to_excisive_functor}.

In all interesting situations one can take a global point of view. Namely, one starts with a covariant
functor respecting equivalences $\bfE \colon \Groupoids \to \Spectra$ and defines for a
group $G$ the functor $\bfE^G$ to be the composite of $\bfE$ with the functor
$\OrG \to \Groupoids$ given by the transport groupoid of a $G$-set, see
Subsection~\ref{subsec:Spectra_over_Groupoids}.


\subsection{Isomorphism Conjectures}
\label{subsec:Isomorphism_Conjectures_intro}

The Meta Isomorphism Conjecture for $G$, $\calf$ and $\bfE^G$,
see  Section~\ref{sec:The_Meta-Isomorphism_Conjecture}, says that the assembly map
of~\eqref{asmb_map_hocolim_intro} is a weak homotopy equivalence,
or, equivalently, that the map~\eqref{asmb_homologically_intro} is bijective for all $n \in \IZ$.

If we take for $\bfE$ an appropriate functor modelling the algebraic $K$-theory or the
algebraic $L$-theory with decoration $\langle -\infty \rangle$ of the group ring $RG$ and
for $\calf$ the family of virtually cyclic subgroups, we obtain the Farrell-Jones
Conjecture~\ref{con:FJC}. It assembles $K_n(RG)$ and $L_n^{\langle -\infty \rangle}(RG)$
in terms of $K_n(RH)$ and $L_n^{\langle -\infty \rangle}(RH)$, where $H$ runs through the
virtually cyclic subgroups of $G$.

If we take for $\bfE$ an appropriate functor modelling the topological $K$-theory of the
reduced group $C^*$-algebra $C^*_r(G)$ and for $\calf$ the family of finite subgroups, we
obtain the Baum-Connes Conjecture~\ref{con:BCC}. It assembles $K_n^{\topo}(C^*_r(G))$ in
terms of $K_n^{\topo}(C^*_r(H))$, where $H$ runs through the finite subgroups of $G$.

The Farrell-Jones Conjecture~\ref{con:FJC} and the  Baum-Connes Conjecture~\ref{con:BCC}
are very powerful conjectures and are the main motivation for the study of assembly maps.
A survey of a lot of striking applications such as the ones to the conjectures of
Bass, Borel, Gromov-Lawson-Rosenberg, Kadison, Kaplansky, and Novikov
is given in Subsections~\ref{subsec:Applications_of_the_Farrell-Jones_Conjecture}
and~\ref{subsec:Applications_of_the_Baum-Connes_Conjecture}. 
The Farrell-Jones Conjecture~\ref{con:FJC} and the  Baum-Connes Conjecture~\ref{con:BCC} are 
known to be true for a surprisingly  large class of groups, as explained in 
Subsections~\ref{subsec:The_status_of_the_Farrell-Jones_Conjecture}
and~\ref{subsec:The_status_of_the_Baum-Connes_Conjecture}.
All of this is an impressive example how homotopy theoretic methods can be used
for problems in other fields such as algebra, geometry, manifold theory and operator algebras.


\subsection{Other interpretations of  assembly maps}
\label{subsec:Other_interpretations_of_the_assembly_map_intro}

The homotopy theoretic approach is the best for structural purposes.  The applications and
the proofs of the Farrell-Jones Conjecture~\ref{con:FJC} and the Baum-Connes
Conjecture~\ref{con:BCC} require sophisticated analytic, topological and geometric
interpretations of the homotopy theoretic assembly maps, for instance in terms of surgery
theory, see
Subsection~\ref{subsec:The_interpretation_of_the_Farrell-Jones_assembly_map_for_L-theory_in_terms_of_surgery_theory},
as forget control maps, see
Subsection~\ref{subsec:The_interpretation_of_the_Farrell-Jones_assembly_map_in_terms_of_controlled_topology},
and in terms of index theory, see
Subsection~\ref{subsec:The_interpretation_of_the_Baum-Connes_assembly_map_in_terms_of_index_theory}.
This presents an intriguing interaction between homotopy theory, geometry and operator
theory.


\subsection{The universal property of the assembly map}
\label{subsec:The_universal_property_of_the_assembly_map_intro}

In Section~\ref{sec:The_universal_property} we characterize the assembly map in the sense
that it is the universal approximation from the left by an excisive functor of a given
homotopy invariant functor $\GCWpairs \to \Spectra$.  This is the key ingredient in the
difficult identification of the various assembly maps mentioned in
Subsection~\ref{subsec:Other_interpretations_of_the_assembly_map_intro} above.  It
reflects the fact that in all of the Isomorphism Conjectures the hard and interesting
object is the target and the source is given by the $G$-homology of the classifying
$G$-spaces for a specific family of subgroups. The source is more accessible than the
target since one can apply standard methods from algebraic topology such as spectral
sequences and equivariant Chern characters.


\subsection{Relative assembly maps}
\label{subsec:Relative_assembly_maps_intro}

Relative assembly maps are studied in Section~\ref{sec:Relative_assembly_maps}.  They
address the problem to make the families appearing in the various Isomorphism Conjectures as
small as possible.


\subsection{Further aspects of assembly maps}
\label{subsec:Further_aspects_of_assembly_maps_intro}

The homotopy theoretic approach to assembly allows to relate assembly maps for various theories, 
such as the algebraic $K$-theory and $A$-theory via linearization, see 
Subsection~\ref{subsec:Relating_the_assembly_maps_for_K-theory_and_for_A-theory},
algebraic  $K$-theory of groups rings and topological $K$-theory of reduced group $C^*$-algebras, see
Subsection~\ref{subsec:Ralating_the_assembly_maps_of_Farrell-Jones_to_the_one_of_Baum-Connes},
and algebraic  $K$-theory of groups rings and the topological cyclic homology of the spherical group ring via cyclotomic traces,
see Subsection~\ref{subsec:Relating_the_assembly_maps_of_Farrell-Jones_to_the_one_the_one_for_topological_cyclic_homology}.
How assembly maps can be used for computations is illustrated in 
Section~\ref{sec:Computationally_tools} which is based on the global point of view
described in Section~\ref{sec:The_global_point_of_view}. Finally we formulate the challenge
of extending equivariant homotopy theory for finite groups to infinite groups in
Section~\ref{sec:The_challenge_of_extending_equivariant_homotopy_to_infinite_groups}.

The idea of the geometric assembly map  is due to Quinn~\cite{Quinn(1971),Quinn(1995a)}
and its algebraic counterpart was introduced by Ranicki~\cite{Ranicki(1992)}.


\subsection{Conventions}
\label{subsec:Conventions_intro}

Throughout this paper $G$ denotes a (discrete) group. Ring means associative ring with unit.
All spectra are non-connective.


\subsection{Acknowledgments}
\label{subsec:Acknowledgements_intro}

This paper is dedicated to Andrew Ranicki. It is financially supported by the ERC Advanced Grant
``KL2MG-interactions'' (no.  662400) of the author granted by the
European Research Council, and by the Cluster of Excellence
``Hausdorff Center for Mathematics'' at Bonn.

The paper is organized as follows: 
\tableofcontents


\typeout{------------------------------- Some basic categories     --------------------------------}

\section{Some basic categories}
\label{sec:Some_basic_categories}

In this section we recall  some well-known basic categories.


\subsection{$G$-$CW$-complexes}
\label{subsec:G-CW-complexes}

\begin{definition}[$G$-$CW$-complex]
\label{def:G-CW-complex}
A \emph{$G$-$CW$-complex}%
\index{G-CW-complex@$G$-$CW$-complex}
$X$ is a $G$-space together with a $G$-invariant filtration
$$\emptyset = X_{-1} \subseteq X_0 \subset X_1 \subseteq \ldots \subseteq
X_n \subseteq \ldots \subseteq\bigcup_{n \ge 0} X_n = X$$
such that $X$ carries the colimit topology with respect to this filtration
(i.e., a set $C \subseteq X$ is closed if and only if $C \cap X_n$ is closed
in $X_n$ for all $n \ge 0$) and $X_n$ is obtained
from $X_{n-1}$ for each $n \ge 0$ by attaching
equivariant $n$-dimensional cells, i.e., there exists
a $G$-pushout
\[
\xymatrix@!C=8em{\coprod_{i \in I_n} G/H_i \times S^{n-1} 
\ar[r]^-{\coprod_{i \in I_n} q_i^n} \ar[d]^-{} &
X_{n-1} \ar[d]^-{}  \\
\coprod_{i \in I_n} G/H_i \times D^{n}
\ar[r]^-{\coprod_{i \in I_n} Q_i^n} &
X_n
}
\]
\end{definition}

A map $f \colon X \to Y$ between $G$-$CW$-complexes is called \emph{cellular}
if $f(X_n) \subseteq Y_n$ holds for all $n \ge 0$. We denote by
$\GCW$ the category of $G$-$CW$-complexes with cellular $G$-maps as morphisms and by
$\GCWpairs$ the corresponding category of $G$-$CW$-pairs.
For basic information about $G$-$CW$-complexes we refer for instance
to~\cite[Chapter~1 and~2]{Lueck(1989)}.


\subsection{The orbit category}
\label{subsec:orbit_category}
The orbit category $\OrG$ has as objects homogeneous spaces $G/H$ and as morphisms
$G$-maps.  It can be viewed as the category of $0$-dimensional $G$-$CW$-complexes, whose
$G$-quotient space is connected. In particular we can think of $\OrG$ as a full
subcategory of $\GCW$.


\subsection{Spectra}
\label{subsec:spectra}

In this paper we can work with the most elementary category $\Spectra$ of spectra.
A \emph{spectrum}%
\index{spectrum}
$\mathbf{E} = \{(E(n),\sigma(n)) \mid n \in \IZ\}$ is a sequence of pointed spaces
$\{E(n) \mid n \in \IZ\}$ together with pointed maps
called \emph{structure maps}
$\sigma(n) \colon  E(n) \wedge S^1 \longrightarrow E(n+1)$.
A \emph{map of spectra}
$\bff \colon  \bfE \to \bfE^{\prime}$ is a sequence of maps
$f(n) \colon  E(n) \to E^{\prime}(n)$
which are compatible with the structure maps $\sigma(n)$, i.e., we have
$f(n+1) \circ \sigma(n)  = 
\sigma^{\prime}(n) \circ \left(f(n) \wedge \id_{S^1}\right)$
for all $n \in \IZ$.
Maps of spectra are sometimes called functions in the literature, they should not be confused 
with the notion of a map of spectra in
the stable category, see~\cite[III.2.]{Adams(1974)}. 

The homotopy groups of a spectrum%
\index{homotopy groups of a spectrum}
are defined by 
\begin{eqnarray}
\pi_i(\mathbf{E})%
& := &
\colim_{k \to \infty} \pi_{i+k}(E(k)),
\label{homotopy_groups_of_spectra}
\end{eqnarray}
where the $i$th structure map of the system $\pi_{i+k}(E(k))$ is given by the composite
\[
\pi_{i+k}(E(k)) \xrightarrow{S} \pi_{i+k+1}(E(k)\wedge S^1)
\xrightarrow{\sigma (k)_*} \pi_{i+k+1}(E(k +1))
\]
of the suspension homomorphism $S$ and the
homomorphism induced by the structure map.
A \emph{weak equivalence} of spectra is a map $\bff\colon \bfE \to \bfF$ 
of spectra inducing an isomorphism on all homotopy
groups.


\typeout{----------------- $G$-homology theories and $\Or(G)$-spectra -----------------------------}

\section{$G$-homology theories and $\Or(G)$-spectra}
\label{sec:G-homology_theories_and_Or(G)-spectra}

Let $\Lambda$ be a commutative ring. Next we recall the obvious generalization of the notion of
a (generalized) homology theory to a $G$-homology theory.

\begin{definition}[$G$-homology theory]
\label{def:G-homology_theory}
A \emph{$G$-homology theory $\calh_*^G$%
\index{G-homology theory@$G$-homology theory}
with values
in $\Lambda$-modules}  is a collection of
covariant functors 
$\calh^G_n$
from the category $\GCWpairs$ of $G$-$CW$-pairs to the category of
$\Lambda$-modules indexed by $n \in \IZ$ together with natural transformations
\[
\partial_n^G(X,A)\colon  \calh_n^G(X,A) \to
\calh_{n-1}^G(A):= \calh_{n-1}^G(A,\emptyset)
\]
for $n \in \IZ$
such that the following axioms are satisfied:
\begin{itemize}

\item \emph{$G$-homotopy invariance} \\[1mm]
If $f_0$ and $f_1$ are $G$-homotopic $G$-maps of $G$-$CW$-pairs 
$(X,A) \to (Y,B)$, then $\calh_n^G(f_0) = \calh^G_n(f_1)$ for $n \in \IZ$;

\item \emph{Long exact sequence of a pair}\\[1mm]
Given a pair $(X,A)$ of $G$-$CW$-complexes,
there is a long exact sequence
\begin{multline*}
\quad \quad \quad \ldots \xrightarrow{\calh^G_{n+1}(j)}
\calh_{n+1}^G(X,A) \xrightarrow{\partial_{n+1}^G}
\calh_n^G(A) \xrightarrow{\calh^G_{n}(i)} \calh_n^G(X)
\\
\xrightarrow{\calh^G_{n}(j)} \calh_n^G(X,A)
\xrightarrow{\partial_n^G} \ldots,
\end{multline*}
where $i\colon  A \to X$ and $j\colon X \to (X,A)$ are the inclusions;

\item \emph{Excision} \\[1mm]
Let $(X,A)$ be a $G$-$CW$-pair and let
$f\colon  A \to B$ be a cellular $G$-map of
$G$-$CW$-complexes. Equip $(X\cup_f B,B)$ with the induced structure
of a $G$-$CW$-pair. Then the canonical map
$(F,f)\colon  (X,A) \to (X\cup_f B,B)$ induces an isomorphism
\[
\calh_n^G(F,f)\colon  \calh_n^G(X,A) \xrightarrow{\cong}
\calh_n^G(X\cup_f B,B);
\]

\item \emph{Disjoint union axiom}\\[1mm]
Let $\{X_i\mid i \in I\}$ be a family of
$G$-$CW$-complexes. Denote by
$j_i\colon  X_i \to \coprod_{i \in I} X_i$ the canonical inclusion.
Then the map
\[
\bigoplus_{i \in I} \calh^G_{n}(j_i)\colon  \bigoplus_{i \in I} \calh_n^G(X_i)
\xrightarrow{\cong} \calh_n^G\left(\coprod_{i \in I} X_i\right)
\]
is bijective.

\end{itemize}
\end{definition}

If $\bfE$ is a spectrum, then one gets  a (non-equivariant) homology
theory $H_*(-;\bfE)$ by defining
\[
  H_n(X,A;\bfE) = \pi_n\left((X_+ \cup_{A_+} \cone(A_+)) \wedge \bfE
  \right)
\]
for a $CW$-pair $(X,A)$ and $n \in \IZ$, where $X_+$ is obtained from $X$ by adding a
disjoint base point and $\cone$ denotes the (reduced) mapping cone.  Its main property is
$H_n(\pt;\bfE) = \pi_n(\bfE)$. This extends to $G$-homology theories as follows. Since the
building blocks of $G$-spaces are homogeneous spaces, we will have to consider 
a covariant $\Or(G)$-spectrum, i.e., a covariant functor $\bfE^G \colon \Or(G) \to \Spectra$,
instead of a  spectrum.

\begin{definition}[Excisive]
\label{def:excisive}
We call a covariant functor
\[
\bfE \colon  \GCWpairs \to \Spectra
\]
\emph{homotopy invariant}
if it sends $G$-homotopy equivalences to
weak homotopy equivalences of spectra.

The functor $\bfE$ is
\emph{excisive}%
\index{excisive}
if it has the following four properties:

\begin{itemize}
\item It is  homotopy invariant;

\item The spectrum $\bfE (\emptyset)$ is weakly contractible;

\item It respects homotopy pushouts up to
weak homotopy equivalence, i.e.,
if the $G$-$CW$-complex $X$ is the union of
$G$-$CW$-subcomplexes $X_1$ and $X_2$ with
intersection $X_0$, then the canonical
map from the homotopy pushout of
$\bfE(X_2) \longleftarrow \bfE(X_0)
\longrightarrow \bfE(X_2)$  to
$\bfE(X)$ is a weak homotopy equivalence of spectra;

\item It respects 
disjoint unions up to weak homotopy, i.e., the natural map
$\bigvee_{i \in I} \bfE(X_i) \to \bfE(\coprod_{i \in I} X_i)$
is a  weak homotopy equivalence for all
index sets $I$. 
\end{itemize}
\end{definition}

One easily checks

\begin{lemma}\label{lem:G_homolog_theory_coming_from_excisive_spectrum}
Suppose that the covariant functor $\bfE \colon  \GCWpairs \to \Spectra$ is excisive.
Then we obtain a $G$-homology theory with values in $\IZ$-modules
 by assigning to $G$-$CW$-pair $(X,A)$ and $n \in \IZ$ the abelian group
$\pi_n(\bfE(X,A))$.
\end{lemma}

A $G$-space $X$ defines a contravariant $\Or(G)$-space
$O^G(X)$ by sending $G/H$ to $\map_G(G/H,X) = X^H$.
Given a contravariant pointed $\Or(G)$-space $Y$ and a covariant pointed $\Or(G)$-space
$Z$, there is the pointed space $Y \wedge_{\Or(G)} Z$.  Its construction is explained for
instance in~\cite[Section~1]{Davis-Lueck(1998)}. This construction is natural in $Y$ and $Z$.  Its
main property is that one obtains for every pointed space $X$ an adjunction homeomorphism
\[
\map(Y \wedge_{\Or(G)} Z,X) \xrightarrow{\cong} \mor_{\Or(G)}(Y,\map(Z,X))
\]
where the source is the pointed mapping space and the target is the topological space of
natural transformations from $Y$ to the contravariant pointed $\Or(G)$-space $\map(Z,X)$
sending $G/H$ to the pointed mapping space $\map(Z(G/H),X)$.  If $\bfE^G$ is a covariant
$\Or(G)$-spectrum, then one obtains a spectrum $Y \wedge_{\Or(G)} \bfE^G$. Hence we can
extend a covariant functor $\bfE^G \colon \OrG \to \Spectra$ to a covariant functor
\begin{equation}
(\bfE^G)_{\%} \colon \GCWpairs \to \Spectra, \quad (X,A) 
\mapsto O^G(X_+ \cup_{A_+} \cone(A_+)) \wedge_{\Or(G)} \bfE^G.
\label{E_procent}
\end{equation}

The easy proofs of the following two results are left to the reader.

\begin{lemma}\label{lem:extending_from_Or(G)_to_excisive_functor}
If $\bfE^G$ is a covariant $\Or(G)$-spectrum, then $(\bfE^G)_{\%}$ is excisive and we obtain a 
$G$-homology theory $H_*^G(-;\bfE^G)$ by
\[
H_n^G(X,A;\bfE^G) = \pi_n((\bfE^G)_{\%}(X,A)) = \pi_n\bigl(O^G(X_+ \cup_{A_+} \cone(A_+)) \wedge_{\Or(G)} \bfE^G\bigr)
\]
satisfying $H_n^G(G/H;\bfE^G) = \pi_n(\bfE^G(G/H))$ for $n \in \IZ$ and $H \subseteq G$.
\end{lemma}

\begin{lemma}\label{lem:transformation_of_G-homology_theories} Let
  $\bft \colon \bfE \to \bfF$ be a natural transformation of covariant functors
  $\GCW^2 \to \Spectra$.  Suppose that $\bfE$ and $\bfF$ are excisive and $\bft(G/H)$ is a
  weak homotopy equivalence for any homogeneous $G$-space $G/H$.

  Then $\bft(X,A) \colon \bfE(X,A) \to \bfF(X,A)$ is a weak homotopy equivalence for every
  $G$-$CW$-pair $(X,A)$.
\end{lemma}


\typeout{------------------------------- Approximation by an excisive functor --------------------------------}

\section{Approximation by an excisive functor}
\label{sec:Approximation_by_an_excisive_functor}

The following result follows from~\cite[Theorem~6.3]{Davis-Lueck(1998)}. Its non-equivariant version is due to
Weiss-Williams~\cite{Weiss-Williams(1995a)}.

\begin{theorem}[Approximation by an excisive functor] 
\label{the:approximation_by_an_excisive_functor}
\index{approximation by an excisive functor}
Let $\bfE \colon \GCWpairs \to \Spectra$ be a covariant functor which is homotopy
invariant.  Let $\bfE| \colon \OrG \to \Spectra$ be its composite with the obvious
inclusion $\OrG \to \GCWpairs$.

Then there exists a covariant functor
\[
\bfE^{\%} \colon \GCWpairs \to \Spectra
\]
and natural transformations
\begin{eqnarray*} 
\bfA_{\bfE} \colon \bfE^{\%}  & \to & \bfE;
\\
\bfB_{\bfE} \colon \bfE^{\%}  & \to & \bfE|_{\%},
\end{eqnarray*}
satisfying:

\begin{enumerate}

\item\label{the:approximation_by_an_excisive_functor;exercisive}
The functor $\bfE^{\%}$ is excisive;

\item\label{the:approximation_by_an_excisive_functor:upper} 
The map   $\bfA_{\bfE}(G/H) \colon \bfE^{\%}(G/H) \to \bfE(G/H)$ is a weak homotopy equivalence
  for every homogeneous space $G/H$;

\item\label{the:approximation_by_an_excisive_functor:lower}
The map $\bfB_{\bfE}(X,A) \colon \bfE^{\%}(X,A)   \to \bfE|_{\%}(X,A)$ is a weak homotopy equivalence
for every $G$-$CW$-pair $(X,A)$;

\item\label{the:approximation_by_an_excisive_functor:charcterization}
The functor $\bfE$ is excisive if and only if $\bfA_{\bfE}(X,A)$ is a weak homotopy equivalence 
for every $G$-$CW$-pair $(X,A)$;

\item\label{the:approximation_by_an_excisive_functor:naturality}
The transformations $\bfA_{\bfE}$ and $\bfB_{\bfE}$ are functorial in $\bfE$.

\end{enumerate}
\end{theorem}

Although one does not need to understand the explicite construction of $\bfE^{\%}$,
$\bfA_{\bfE}$ and $\bfB_{\bfE}$ and the proof of
Theorem~\ref{the:approximation_by_an_excisive_functor} for the applications of
Theorem~\ref{the:approximation_by_an_excisive_functor}  and for the reminder of
this paper, we make some comments about it for the interested reader.

As an illustration we firstly present a naive suggestion in the non-equivariant case,
which turns out to require too restrictive assumptions on $\bfE$ and  therefore will not be
the final solution, but conveys a first idea. Namely, we can define a map of pointed sets
$X_+ \wedge \bfE(\pt) \to \bfE(X)$ for a $CW$-complex $X$ by sending an element in the
target represented by $(x,e)$ for $x \in X$ and $e \in \bfE(\pt)$ to
$\bfE(c_x \colon \pt \to X)(e)$, where $c_x \colon \pt \to X$ is the constant map with
value $x$.  The problem is that the only reasonable way of ensuring the continuity of
this map is to require that $\bfE$ itself is continuous, i.e., the map
$\map(X,Y) \to\map(E(X)_n,E(Y)_n)$ sending $f$ to $E(f)_n$ has to be continuous for all
$n \in \IZ$. But this assumption is not satisfied for the functors $\bfE$ which are of
interest for us and will be considered below.

The solution is to take homotopy invariance into account and to work simplicially.
Let us consider the special case, where $G$ is trivial and $X$ is a simplicial complex.
For any simplex $\sigma$ of $X$ we have the inclusion $i[\sigma] \colon \sigma \to X$
and can therefore define maps
\begin{eqnarray*}
&
A_{\bfE}[\sigma]_n \colon \sigma_+ \wedge E(\sigma)_n 
\xrightarrow{\pr_+ \wedge \id_{E(\sigma)_n}} \pt_+ \wedge E(\sigma)_n
= E(\sigma)_n \xrightarrow{E(i[\sigma])_n} E(X)_n;
&
\\
&
B_{\bfE}[\sigma]_n \colon \sigma_+ \wedge E(\sigma)_n 
\xrightarrow{\id_{\sigma_+} \wedge E(\pr)_n} \sigma_+ \wedge E(\pt)_n,
&
\end{eqnarray*}
where $\pr$ denotes the projection onto $\pt$.
Now define a space $E^{\%}(X)_n$ by glueing the spaces $\sigma_+ \wedge E(\sigma)_n$ for
$\sigma$ running over the simplices of $X$ together according to the simplicial structure,
more precisely, for an inclusion $j \colon \tau \to \sigma$ of simplices we identify a
point in $\tau_+ \wedge E(\tau)_n$ with its image in $\sigma_+ \wedge E(\sigma)_n$ under
the obvious map $j_+ \wedge E(j)_n$. One easily checks that the various maps $A_{\bfE}[\sigma]_n$
and $B_{\bfE}[\sigma]_n$ fit together to maps of pointed spaces
\begin{eqnarray*}
A_{\bfE}(X)_n \colon E^{\%}(X)_n  & \to & E(X)_n;
\\
B_{\bfE}(X)_n \colon E^{\%}(X)_n & \to & E|_{\%}(X)_n := X_+ \wedge E(\pt)_n,
\end{eqnarray*}
and thus to maps of spectra
\begin{eqnarray*}
\bfA_{\bfE}(X) \colon \bfE^{\%}(X)  & \to & \bfE(X)_n;
\\
\bfB_{\bfE}(X) \colon \bfE^{\%}(X) & \to & \bfE|_{\%}(X) := X_+ \wedge \bfE(\pt).
\end{eqnarray*}
Notice that each map $\bfE(\pr) \colon \bfE(\sigma) \to \bfE(\pt)$ is by assumption a weak
homotopy equivalence.  This implies that
$\bfB_{\bfE}(X) \colon \bfE^{\%}(X) \to \bfE|_{\%}(X)$ is a weak homotopy equivalence.
Since the functor $\bfE_{\%}$ is excisive, the functor $\bfE^{\%}$ is excisive.  If
$X = \pt$, the map $\bfA_{\bfE}(\pt) \colon \bfE^{\%}(\pt) \to \bfE(\pt)$ is an
isomorphism and in particular a weak homotopy equivalence.

Now we see, where the name assembly map comes from. In the case of a
simplicial complex $X$ we want to assemble $\bfE(X)$ by its values $\bfE(\sigma)$ for the
various simplices of $X$, which leads to the definition of $\bfE^{\%}(X)$. Intuitively
it is clear that $\bfE(X)$ carries the same information as $\bfE|^{\%}(X)$ if and only
if $\bfE$ is excisive since the condition excisive allows to compute the values of $\bfE$
on $X$ by its values on the simplices taking into account how the simplices are glued
together to yield $X$.

Finally one wants a definition that is independent of the simplicial structure and
actually applies to more general spaces $X$ than simplicial complexes. Therefore one uses
simplicial sets and in particular the singular simplicial set $S.X$, Recall that $S.X$ is the functor
from the category of finite ordered sets ${\boldmath \Delta}$ to the category of sets
$\Sets$ sending the finite ordered set $[p]$ to the set $\map(\Delta_p,Y)$ for $\Delta_p$
the standard $p$-simplex.  For the equivariant version one has to bring the orbit category
into play. So one considers for a $G$-space $X$ the functor
\[
\OrG \times {\boldmath \Delta} \to \Sets, \quad ( G/H,[p]) \mapsto \map_G(G/H \times \Delta_p,X).
\]
In some sense on uses free resolution of contravariant functors
$\OrG \times {\boldmath \Delta} \to \Spaces$ to get the right construction of $\bfE^{\%}$
and of the desired transformations $\bfA_{\bfE}$ and $\bfB_{\bfE}$ so that the claims
appearing in Theorem~\ref{the:approximation_by_an_excisive_functor} can be proved.
Details can be found in~\cite{Davis-Lueck(1998)}.


\typeout{-------------------------------  The universal property   --------------------------------}

\section{The universal property}
\label{sec:The_universal_property}

\index{assembly map!universal property}
Next we explain why Theorem~\ref{the:approximation_by_an_excisive_functor}
characterizes the assembly map in the sense that
$\bfA_{\bfE} \colon  \bfE^{\%} \longrightarrow \bfE$
is the universal approximation from
the left by an excisive
functor  of a homotopy invariant functor $\bfE \colon \GCWpairs \to \Spectra$.
Namely, let $\bfT \colon  \bfF \to  \bfE$
be a transformation of  covariant  functors $\GCWpairs \to \Spectra$
such that $\bfF$ is  excisive. Then for any
$G$-$\calf$-$CW$-pair  $(X,A)$ the following diagram commutes
\[
\xymatrix@!C=9em{%
\bfF^{\%}(X) \ar[r]^{\bfA_{\bfF}(X)}_{\simeq} \ar[d]_{\bfT^{\%}(X)}
&
\bfF(X) \ar[d]^{\bfT(X)}
\\
\bfE^{\%}(X) \ar[r]_{\bfA_{\bfE}(X)} 
&
\bfE(X)
}
\]
and $\bfA_{\bfF}(X)$ is a weak homotopy equivalence by
Theorem~\ref{the:approximation_by_an_excisive_functor}~%
\eqref{the:approximation_by_an_excisive_functor:charcterization}. 
Hence $\bfT (X)$ factorizes over $\bf A_{\bfE}(X)$ up to natural weak homotopy equivalence.

Suppose additionally that $\bfT(G/H)$ is a weak homotopy equivalence for every subgroup
$H \subseteq G$.  Then both $\bfT^{\%}(X)$ and $\bfA_{\bfF}(X)$ are weak homotopy
equivalences by Lemma~\ref{lem:transformation_of_G-homology_theories} and
Theorem~\ref{the:approximation_by_an_excisive_functor}~%
\eqref{the:approximation_by_an_excisive_functor:charcterization}, and hence $\bfT (X)$ can be
identified with $\bfA_{\bfE}(X)$ up to natural weak homotopy equivalence.

Recall that there is a natural weak equivalence $\bfB_{\bfE}(X) \colon \bfE^{\%}(X) \xrightarrow{\simeq} \bfE|_{\%}(X)$,
so that one may replace in the considerations above $\bfE^{\%}(X)$ by  $\bfE|_{\%}(X)$, which depends
on the values of $\bfE$ on homogeneous spaces only. 
This universal property  will be the key ingredient
for the identification of  various versions of assembly maps.


\typeout{---------------- Classifying spaces for families of subgroups --------------------------------}

\section{Classifying spaces for families of subgroups}
\label{sec:Classifying_spaces_for_families_of_subgroups}

We recall the notion classifying space for a family which was introduced by tom Dieck~\cite{Dieck(1972)}.

\begin{definition}[Family of subgroups]
\label{def:family_of_subgroups}
A \emph{family $\calf$ of subgroups of a group $G$}%
\index{family of subgroups} 
is a set of subgroups of $G$ which is
closed under conjugation with elements of $G$ and under passing to subgroups.
\end{definition}

Our main examples of families are the trivial family $\caltr$ consisting of the trivial
subgroup, the family $\calall$ of all subgroups, and the families $\calfcyc$,
$\calcyc$, $\calfin$, and $\calvcyc$ of finite cyclic subgroups, of cyclic subgroups, of
finite subgroups, and of virtually cyclic subgroups.

\begin{definition}[Classifying $G$-space for a family of subgroups]
\label{def:Classifying_G-space_for_a_family_of_subgroups}
Let $\calf$ be a family of subgroups of $G$. A model
$\EGF{G}{\calf}$
for the \emph{classifying spaces for the family $\calf$ of subgroups of $G$}%
\index{classifying spaces for a family of subgroups}
is a $G$-$CW$-complex $\EGF{G}{\calf}$ that has the following properties:
\begin{enumerate}
\item All isotropy groups of $\EGF{G}{\calf}$ belong to $\calf$;
\item For any $G$-$CW$-complex $Y$, whose isotropy groups belong to $\calf$, 
there is up to $G$-homotopy precisely one $G$-map $Y \to X$. 
\end{enumerate}

We abbreviate $\eub{G} := \EGF{G}{\calfin}$
and call it the \emph{universal $G$-space  for proper $G$-actions}.
We also write  $\edub{G} := \EGF{G}{\calvcyc}$.
\end{definition}

Equivalently, $\EGF{G}{\calf}$ is a terminal object in the $G$-homotopy category of
$G$-$CW$-complexes, whose isotropy groups belong to $\calf$. 
In particular two models for $\EGF{G}{\calf}$ are $G$-homotopy equivalent
and for two families $\calf_0 \subseteq \calf_1$ there is
up to $G$-homotopy precisely one $G$-map $\EGF{G}{\calf_0} \to \EGF{G}{\calf_1}$.
There are  functorial constructions for $\EGF{G}{\calf}$
generalizing the bar construction, see~\cite[Section~3 and Section~7]{Davis-Lueck(1998)}.

\begin{theorem} [Homotopy characterization of $\EGF{G}{\calf}$]
\label{the:G-homotopy_characterization_of_EGF(G)(calf)}
   A $G$-$CW$-complex $X$ is a model for $\EGF{G}{\calf}$ if and only if for every subgroup
  $H \subseteq G$ its $H$-fixed point set $X^H$ is weakly contractible if $H \in \calf$,
  and is empty if $H \notin \calf$.
\end{theorem}

A model for $\EGF{G}{\calall}$ is $G/G$. A model for $\EGF{G}{\caltr}$ is the same as a
model for $EG$ i.e, the universal covering of $BG$, or, equivalently, the total space of
the universal $G$-principal bundle. There are many interesting geometric models for
classifying spaces $\eub{G} = \EGF{G}{\calfin}$, e.g., the Rips complex for a hyperbolic
group, the Teichm\"uller space for a mapping class group, and so on. The question whether
there are finite-dimensional models, models of finite type or finite models has been
studied intensively during the last decades. For more information about classifying spaces for families we
refer for instance to~\cite{Lueck(2005s)}.


\typeout{------------------------ The Meta-Isomorphism Conjecture -------------------------------}

\section{The Meta-Isomorphism Conjecture}
\label{sec:The_Meta-Isomorphism_Conjecture}

In this section we formulate the  Meta-Isomorphism Conjecture, from which all other Isomorphism
Conjectures such as the one due to Farrell-Jones and Baum-Connes are obtained by specifying
the parameters $\bfE$ and $\calf$.


\subsection{The Meta-Isomorphism Conjecture for $G$-homology theories}
\label{subsec:The_Meta-Isomorphism_Conjecture_for_G-homology_theories}

Let $\calh^G_*$ be a $G$-homology theory with values in
$\Lambda$-modules for some commutative ring  $\Lambda$. 
The projection $\pr \colon \EGF{G}{\calf} \to G/G$ induces for all integers $n \in \IZ$ a
homomorphism of $\Lambda$-modules
\begin{equation}
\calh_n^G(\pr) \colon \calh^G_n(\EGF{G}{\calf}) \to \calh_n^G(G/G)
\label{assembly_map_for_G-homology_theories}
\end{equation}
which is  called the  \emph{assembly map}.%
\index{assembly map!for $G$-homology theories}

\begin{conjecture}[Meta-Isomorphism Conjecture for $G$-homology theories] 
\label{con:Meta_Isomorphisms_Conjecture_for_G-homology_theories}
The group $G$ satisfies the \emph{Meta-Isomorphism Conjecture}
\index{Meta-Isomorphism Conjecture!for $G$-homology theories}
with respect to the $G$-homology theory $\calh^G_*$ and the family $\calf$ of subgroups of $G$,
if the assembly map 
\[
\calh_n^G(\pr) \colon \calh^G_n(\EGF{G}{\calf}) \to \calh_n^G(G/G)
\]
of~\eqref{assembly_map_for_G-homology_theories} is bijective for all $n \in \IZ$.
\end{conjecture}

If we choose $\calf$ to be the family $\calall$ of all subgroups, then $G/G$ is a model
for $\EGF{G}{\calall}$ and the Meta-Isomorphism
Conjecture~\ref{con:Meta_Isomorphisms_Conjecture_for_G-homology_theories} is obviously
true.  The point is to find an as small as possible family $\calf$. The idea of the
Meta-Isomorphism Conjecture~\ref{con:Meta_Isomorphisms_Conjecture_for_G-homology_theories}
is that one wants to compute $\calh^G_n(G/G)$, which is the unknown and the interesting
object, by assembling it from the values $\calh^G_n(G/H)$ for $H \in \calf$.


\subsection{The Meta-Isomorphism Conjecture on the level of spectra}
\label{subsec:The_Meta-Isomorphism_Conjecture_on_the_level_of_spectra}

Often the construction of the assembly map is done already on the level of spectra or can
be lifted to this level.  Consider a covariant functor
\[
\bfE^G  \colon \Or(G) \to \Spectra.
\]

\begin{conjecture}[Meta-Isomorphism Conjecture for spectra]
\label{con:Meta_Isomorphisms_Conjecture_for_Groupoids_spectra}
\index{Meta-Isomorphism Conjecture!for spectra}
The group $G$ satisfies the \emph{Meta-Isomorphism Conjecture} with respect to the
covariant functor $\bfE^G  \colon \Or(G) \to \Spectra$ and the
family $\calf$ of subgroups of $G$, if the projection $\pr \colon \EGF{G}{\calf} \to G/G$
induces a weak homotopy equivalence
\[
(\bfE^G)_{\%}(\pr) \colon (\bfE^G)_{\%}(\EGF{G}{\calf}) \to (\bfE^G)_{\%}(G/G) = \bfE^G(G/G). 
\]
\end{conjecture}

Notice that $(\bfE^G)_{\%}(\pr) \colon (\bfE^G)_{\%}(\EGF{G}{\calf}) \to (\bfE^G)_{\%}(G/G) = \bfE^G(G/G)$ 
is a weak homotopy equivalence
if and only if for every $n \in \IZ$ the map
\[
H_n^G(\pr;\bfE^G) \colon H^G_n(\EGF{G}{\calf};\bfE^G) \to H_n^G(G/G;\bfE^G)
\]
is a bijection, where $H_*^G(-;\bfE^G)$ is the $G$-homology theory associated to $\bfE^G$, see 
Lemma~\ref{lem:extending_from_Or(G)_to_excisive_functor}. In other words,
Conjecture~\ref{con:Meta_Isomorphisms_Conjecture_for_Groupoids_spectra} is equivalent
to Conjecture~\ref{con:Meta_Isomorphisms_Conjecture_for_G-homology_theories} if we take for 
$\calh^G_*$ the $G$-homology theory associated to $\bfE^G$.


\subsection{The assembly map in terms of homotopy colimits}
\label{subsec:The_assembly_map_in_terms_of_homotopy_colimits}

The assembly map appearing in
Conjecture~\ref{con:Meta_Isomorphisms_Conjecture_for_Groupoids_spectra}%
\index{assembly map!in terms of homotopy colimits}
can be interpreted in terms of homotopy colimits as follows.  Let
$\OrGF{G}{\calf}$ be the full subcategory of $\OrG$ consisting of those
objects $G/H$ for which $H$ belongs to $\calf$.  Let
$\bfE^G|_{\OrGF{G}{\calf}}$ be the restriction of $\bfE^G$ to
$\OrGF{G}{\calf}$.  Then we get from the inclusion
$\OrGF{G}{\calf} \to \OrG$ and the fact that $G/G$ is a terminal object
in $\OrG$ a map
\[
\hocolim_{\OrGF{G}{\calf}} \bfE^G|_{\OrGF{G}{\calf}} \to \hocolim_{\OrG} \bfE^G = \bfE^G(G/G).
\]
This map can be identified with
$(\bfE^G)_{\%}(\pr) \colon (\bfE^G)_{\%}(\EGF{G}{\calf}) \to (\bfE^G)_{\%}(G/G)$,
see~\cite[Section~5.2]{Davis-Lueck(1998)}. Again this explains the name assembly map:  we try
to put the values of $\bfE^G$ on homogeneous spaces $G/H$ for $H \in \calf$ together to
get its value at $G/G$.


\subsection{Spectra over $\Groupoids$}
\label{subsec:Spectra_over_Groupoids}

In all interesting cases we will obtain $\bfE^G$ as follows.
Let $\Groupoids$ be the  category  of small groupoids.
Consider a covariant functor 
\[
\bfE \colon \Groupoids \to \Spectra
\]
which \emph{respects equivalences}, i.e., it sends equivalences of groupoids to weak equivalences
of spectra.  Given a $G$-set $S$, its \emph{transport groupoid} $\calT^G(S)$ has $S$ as
set of objects and the set of morphism from $s_0$ to $s_1$ is
$\{g \in G \mid gs_0 = s_1\}$. Composition comes from the multiplication in $G$.  We get
for every group $G$ a functor
\[
\bfE^G  \colon \Or(G) \to \Spectra
\]
by composing $\bfE$ with the functor
$\Or(G) \to \Groupoids, \;G/H \mapsto \calT^G(G/H)$.  

Notice that a group $G$ can be
viewed as a groupoid with one object and $G$ as set of automorphisms of this object and
hence we can consider $\bfE(G)$. We have the obvious identifications
$\bfE(G) = \bfE^G(G/G) = (\bfE^G)_{\%}(G/G)$. Moreover, for every subgroup $H \subseteq G$ 
there is an equivalence of groupoids $H \to \calT^G(G/H)$ sending the unique object of $H$
to the object $eH$, which induces a weak homotopy equivalence  $\bfE(H) \to \bfE^G(G/H)$.

The various prominent Isomorphism Conjectures such as the one due to Farrell-Jones and Baum-Connes
are now obtained by specifying  $\bfE \colon \Groupoids \to \Spectra$, the group $G$  and the family $\calf$.


\typeout{------------------------- The Farrell-Jones Conjecture -------------------------------}

\section{The Farrell-Jones Conjectures}
\label{sec:The_Farrell-Jones_Conjecture}


\subsection{The Farrell-Jones Conjecture for $K$-and $L$-theory}
\label{subsec:The_Farrell-Jones_Conjecture_for_K-_and_L-theory}

Let $R$ be a ring (with involution). There exist covariant functors respecting equivalences
\begin{eqnarray}
\bfK_R%
\colon \Groupoids & \to & \Spectra;
\label{Kalg_GROUPOIDS}
\\
\bfL^{\langle -\infty \rangle}_R \colon \Groupoids & \to & \Spectra,
\label{LGROUPOIDS}
\end{eqnarray}
such that for every group $G$ and all $n \in \IZ$ we have
\begin{eqnarray*}
\pi_n(\bfK_R(G)) & \cong & K_n(RG);
\\
\pi_n(\bfL^{\langle -\infty \rangle}_R(G)) & \cong & L_n^{\langle -\infty \rangle}(RG).
\end{eqnarray*}
Here $K_n(RG)$ is the $n$-th algebraic $K$-group of the group ring $RG$ and
$L_n^{\langle -\infty \rangle}(RG)$ is the $n$th quadratic $L$-group with decoration
$\langle -\infty \rangle$ of the group ring $RG$ equipped with the involution sending
$\sum_{g \in G} r_g g$ to $\sum_{g \in G} \overline{r_g} g^{-1}$.

The details of this construction can be found  in~\cite[Section 2]{Davis-Lueck(1998)}. 
If we now take these functors and the family $\calvcyc$ of virtually cyclic subgroups,
we obtain

\begin{conjecture}[Farrell-Jones Conjecture]\label{con:FJC}
  A group $G$ satisfies the \emph{$K$-theoretic or $L$-theoretic Farrell-Jones Conjecture}%
\index{Farrell-Jones Conjecture!$K$- and $L$-theoretic}
 if for
  every ring (with involution) $R$ the assembly maps induced by the projection
  $\pr \colon \edub{G} \to G/G$
\begin{eqnarray*}
  H_n^G(\pr;\bfK_R) \colon H^G_n(\edub{G};\bfK_R) 
& \to & 
H^G_n(G/G;\bfK_R)  = K_n(RG);
  \\
  H_n^G(\pr;\bfL_R^{\langle -\infty  \rangle}) \colon H^G_n(\edub{G};\bfL_R^{\langle -\infty  \rangle}) 
& \to &  
H^G_n(G/G;\bfL_R^{\langle -\infty  \rangle}) = L_n^{\langle -\infty \rangle}(RG),
\end{eqnarray*}
are bijective for all $n \in \IZ$.
\end{conjecture}

It is crucial that we use  non-connective $K$-spectra and that
the decoration for the $L$-theory is $\langle -\infty \rangle$,
see~\cite{Farrell-Jones-Lueck(2002)}.

The original version of the Farrell-Jones Conjecture appeared in~\cite[1.6 on page~257]{Farrell-Jones(1993a)}. 
A detailed exposition on the Farrell-Jones Conjecture will be given in~\cite{Lueck(2020book)}, 
see also~\cite{Lueck-Reich(2005)}.


\subsection{Applications of the Farrell-Jones Conjecture}
\label{subsec:Applications_of_the_Farrell-Jones_Conjecture}

Here are some consequences of the Farrell-Jones Conjecture. For more information about
these and other applications we refer for instance to
\cite{Bartels-Lueck-Reich(2008appl),Lueck(2020book),Lueck-Reich(2005)}.

\subsubsection{Computations}
One can carry out \emph{explicite computations} of $K$ and $L$-groups of group rings
by applying methods from algebraic topology to the left side given by a $G$-homology theory and
by finding small models for the classifying spaces of families using the  topology and geometry of groups,
see Section~\ref{sec:Computationally_tools}.

\subsubsection{Vanishing of lower and middle $K$-groups}
 If $G$ is a torsionfree group satisfying the $K$-theoretic Farrell-Jones Conjecture~\ref{con:FJC}, 
then $K_n(\IZ G)$ for $n \le -1$, the reduced
projective class group $\widetilde{K}_0(\IZ G)$, and the Whitehead group $\Wh(G)$ vanish.

This has the following consequences. Every homotopy equivalence $f \colon X \to Y$ of
connected $CW$-complexes with $\pi_1(Y) \cong G$ is simple.  Every $h$-cobordism over a
closed manifold $M$ of dimension $\ge 5$ and $G \cong \pi_1(M)$ is trivial.  Every
finitely generated projective $\IZ G$-module is stably free. Every finitely dominated
connected $CW$-complex $X$ with $\pi_1(X) \cong G$ is homotopy equivalent to a finite
$CW$-complex.

\subsubsection{Kaplansky's Idempotent Conjecture}
If the torsionfree group $G$ satisfies the $K$-theoretic Farrell-Jones
  Conjecture~\ref{con:FJC}, then $G$ satisfies the \emph{Idempotent Conjecture}%
\index{Idempotent Conjecture}
  that for a commutative integral domain $R$ the only idempotents of $RG$ are $0$ and $1$.

\subsubsection{Novikov Conjecture}
If $G$ satisfies the $L$-theoretic Farrell-Jones Conjecture~\ref{con:FJC}, then $G$
  satisfies the \emph{Novikov Conjecture}%
\index{Novikov Conjecture}
 about the homotopy invariance of higher
  signatures. For more information about the Novikov Conjecture
 we refer for instance 
to~\cite{Ferry-Ranicki-Rosenberg(1995a),Ferry-Ranicki-Rosenberg(1995b),Kreck-Lueck(2005)}.

\subsubsection{Borel Conjecture}
If $G$ is a torsionfree group satisfying the $K$-theoretic and the $L$-theoretic
  Farrell-Jones Conjecture~\ref{con:FJC}, then $G$ satisfies the \emph{Borel Conjecture}%
\index{Borel Conjecture}
in dimensions $\ge 5$, i.e., if $M$ and $N$ are closed aspherical manifolds of dimension
$\ge 5$ with $\pi_1(M) \cong \pi_1(N) \cong G$, then $M$ and $N$ are homeomorphic and
every homotopy equivalence from $M$ to $N$ is homotopic to a homeomorphism.

\subsubsection{Bass Conjecture}
If $G$ satisfies the $K$-theoretic Farrell-Jones Conjecture~\ref{con:FJC}, then $G$
  satisfies the \emph{Bass Conjecture},%
\index{Bass Conjecture}
 see~\cite{Bartels-Lueck-Reich(2008appl), Bass(1976)}.

\subsubsection{Automorphism groups} The Farrell-Jones Conjecture~\ref{con:FJC} yields
rational computations of the homotopy groups and homology groups of the automorphisms
groups of an aspherical closed manifold in the topological, PL and smooth category, see
for instance instance~\cite{Farrell-Hsiang(1978)},~\cite[Section~2]{Farrell-Jones(1990b)}
and~\cite[Lecture~5]{Farrell(2002)}.  

For instance,
if $M$ is an aspherical orientable closed  (smooth) manifold of dimension $> 10$ with fundamental group $G$
such that $G$ satisfies the Farrell-Jones Conjecture~\ref{con:FJC}, then
we get for $1 \leq i \leq ( \dim M -7 ) / 3$
\[
\pi_i ( \operatorname{Top}( M ) ) \otimes_{\IZ} \IQ 
              = \left\{ \begin{array}{lll} \operatorname{center} ( G ) \otimes_{\IZ} \IQ  & \quad & \mbox{if} \; i=1 ;\\
                                                  0  & \quad & \mbox{if} \; i > 1,
                         \end{array}    
                \right.
\]
and
\[
\pi_i ( \operatorname{Diff} ( M ) ) \otimes_{\IZ} \IQ =
                \left\{ \begin{array}{lll} \operatorname{center} ( G ) \otimes_{\IZ} \IQ  &  & \mbox{if} \; i=1 ;\\
     \bigoplus_{j=1}^{\infty} H_{(i +1) - 4j} ( M ; \IQ ) &  & \mbox{if} \; i > 1 \; \mbox{and} \; \dim M \mbox{ odd} ;\\
                                                      0 & & \mbox{if} \; i > 1 \: \mbox{and} \; \dim M \; \mbox{even}.
                         \end{array}    
                \right.
\]

For a  survey on automorphisms of manifolds we refer to~\cite{Weiss-Williams(2001)}.

\subsubsection{Boundary of hyperbolic groups}
In~\cite{Bartels-Lueck-Weinberger(2010)} a proof of a conjecture of Gromov is given in
dimensions $n \ge 6$ using the Farrell-Jones Conjecture~\ref{con:FJC} that a torsionfree
hyperbolic group with $S^n$ as boundary is the fundamental group of an aspherical closed
topological manifold.  This manifold is unique to homeomorphism. The stable Cannon Conjecture
is treated in~\cite{Ferry-Lueck-Weinberger(2018)}.

\subsubsection{Poincar\'e duality groups}
If $G$ is a Poincar\'e duality group of dimension $n \ge 6$ and satisfies the
Farrell-Jones Conjecture~\ref{con:FJC}, then it is the fundamental group of an aspherical
closed homology ANR-manifold, see~\cite{Bartels-Lueck-Weinberger(2010)}.  It is unique up
to $s$-cobordism. Whether it can be chosen to be an aspherical closed topological
manifold, depends on its Quinn obstruction.

\subsubsection{Tautological classes and aspherical manifolds}
The vanishing of tautological classes for many bundles with fibre an aspherical manifold is proved
in~\cite{Hebestreit-Land-Lueck-Randel-Williams(2017)}.

\subsubsection{Fibering manifolds}
The problem when a map from some closed connected manifold to an
  aspherical closed manifold approximately fibers, i.e., is homotopic to Manifold
  Approximate Fibration, is analyzed in~\cite{Farrell-Lueck-Steimle(2018)}.


\subsection{The interpretation of the Farrell-Jones assembly map for $L$-theory in terms of surgery theory}
\label{subsec:The_interpretation_of_the_Farrell-Jones_assembly_map_for_L-theory_in_terms_of_surgery_theory}
\index{assembly map!in terms of surgery theory}
So far we have given a homotopy theoretic approach to the assembly map. This is the
easiest approach and well-suited for structural questions such as comparing the assembly
maps of various different theories, as explained below. For concrete applications it
is important to give geometric or analytic interpretations. For instance, one key
ingredient in the proof that the Borel Conjecture follows from the Farrell-Jones
Conjecture is a geometric interpretation of the assembly for the trivial family in terms
of surgery theory, notably the  surgery exact sequence, which we briefly sketch next.

\begin{definition}[The  structure set]\label{def:structure_set}
Let $N$ be a closed topological manifold of dimension $n$. We call two simple
homotopy equivalences $f_i\colon  M_i \to N$ from  closed
topological manifolds $M_i$ of dimension $n$ to $N$ for $i = 0,1$ equivalent if there
exists a  homeomorphism  $g\colon  M_0 \to M_1$
such that $f_1 \circ g$ is homotopic to $f_0$. 
\medskip

The \emph{structure set $\cals(N)$}%
\index{structure set} 
of $N$ is the set of equivalence classes of simple
homotopy equivalences $M \to X$ from closed topological manifolds of dimension $n$ to
$N$. This set has a preferred base point, namely, the class of the identity $\id\colon N \to N$.
\end{definition}

One easily checks that the Borel Conjecture holds for $G = \pi_1(N)$ for a closed
aspherical manifold $N$ if and only if $\cals(N)$ consists of precisely one element,
namely, the class of $\id_N \colon N \to N$. The surgery exact sequence, which we will
explain next, gives a way of calculating the structure set.

\begin{definition}[Normal map of degree one]
A \emph{normal map of degree one}%
\index{normal map of degree one}
with target the connected  closed manifold $N$ of dimension $n$  consists of:
\begin{itemize}

\item A   connected closed  $n$-dimensional manifold $M$;

\item A map of degree one $f \colon M \to N$;

\item A $(k+n)$-dimensional vector bundle $\xi$ over $N$;

\item A bundle map $\overline{f} \colon TM \oplus \underline{\IR^k} \to \xi$ covering $f$.

\end{itemize}
\end{definition}

There is an obvious normal bordism relation and we denote by $\caln(N)$ the set of bordism
classes of normal maps with target $N$. One can assign to a normal map $f \colon M \to N$
its surgery obstruction $\sigma(f) \in L_n^s(\IZ G)$ taking values in the $n$th quadratic
$L$-group with decoration $s$, where $G = \pi_1(N)$ and $n = \dim(N)$. If $n \ge 5$, the
surgery obstruction vanishes if and only if one can find (by doing surgery) a representative in the normal
bordisms class, whose underlying map $f$ is a simple homotopy equivalence. It yields a map
$\sigma \colon \caln(N) \to L_n^s(\IZ G)$.  There is a map
$\eta \colon \cals^{\topo}_n(N) \to \caln(N)$ which assigns to the class of  a simple homotopy equivalence
$f \colon M \to N$ with a  closed manifold $M$ as source 
the normal map given by $f$ itself  and the bundle data coming from $TM$ and
$\xi = (f^{-1})^*TM$ for some homotopy inverse $f^{-1}$ of $f$. We denote by
$\caln(N \times [0,1], N \times \partial [0,1])$ the normal bordism classes of normal maps
relative boundary.  Essentially these are normal maps
$(M,\partial M) \to (N \times [0,1], N \times \partial [0,1])$ of degree one which are simple
homotopy equivalences on the boundary. There is a surgery obstruction relative boundary
which yields a map
$\sigma \colon \caln(N \times [0,1], N \times \partial [0,1]) \to L_{n+1}^s(\IZ G)$. There
is a also a map $\partial \colon L_{n+1}^s(\IZ G) \to \cals^{\topo}_n(N)$ which sends an
element $x \in L_{n+1}^s(\IZ G)$ to the class of a simple homotopy equivalence $f \colon M \to N$
for which there exists a normal map relative boundary of triads
$(F,f_0,\id_N) \colon (W;M,N) \to (N \times [0,1]; N \times \{0\},N \times \{1\})$ whose
relative surgery obstruction is $x$.  If $n \ge 5$, then one obtains a long exact sequence of
abelian groups, the \emph{surgery exact sequence}%
\index{surgery exact sequence}
due to Browder, Novikov, Sullivan and Wall
\begin{equation}
\caln(N\times [0,1],N \times \partial [0,1]) 
\xrightarrow{\sigma_{n+1}}  L^s_{n+1}(\IZ G) 
\xrightarrow{\partial} \cals(N) \\
\xrightarrow{\eta} \caln(N) \xrightarrow{\sigma_n}
L_n^s(\IZ G).
\label{surgery_exact_sequence}
\end{equation}
If  we can show that $\sigma_{n+1}$ is surjective and $\sigma_n$ is injective, then the
Borel Conjecture holds for $G = \pi_1(N)$, if $N$ is an aspherical closed manifold of
dimension $n \ge 5$.

Let $\bfL$ be the $L$-theory spectrum. It has the property
$\pi_n(\bfL) = L_n^{\langle -\infty \rangle}(\IZ)$.
Denote by $\bfL\langle 1 \rangle$ its  \emph{$1$-connective cover}. 
It comes with a natural map of spectra $\bfL \langle 1 \rangle \to \bfL$,
which induces on $\pi_i$ an isomorphism for $i \ge 1$, and we have $\pi_i(\bfL\langle 1 \rangle) = 0$ for $i \le 0$.
There are natural identifications coming among other things from the Pontrjagin-Thom construction
\begin{eqnarray*}
u_n \colon \caln(N) 
& \xrightarrow{\cong}  & 
H_n(N;\bfL\langle 1 \rangle) = \pi_n(N_+ \wedge \bfL\langle 1 \rangle);
\\
u_{n+1} \colon \caln(N\times [0,1],N \times \{0,1\}) 
& \xrightarrow{\cong} & 
H_{n+1}(N;\bfL\langle 1 \rangle) = \pi_{n+1}\bigl(N_+ \wedge \bfL\langle 1 \rangle).
\end{eqnarray*}
An easy spectral sequence argument shows that the canonical map
\[
v_n \colon H_n(N;\bfL\langle 1 \rangle) \to H_n(N;\bfL)
\]
is injective and the canonical map
\[
v_{n+1} \colon H_{n+1}(N;\bfL\langle 1 \rangle) \to H_{n+1}(N;\bfL)
\]
is bijective for $n = \dim(N)$.

For the remainder of this subsection we assume additionally that $N$ is aspherical.
There is a natural identification for $m = n,n+1$, see
Definition~\ref{def:equivariant_homology_theory},
\[
w_m \colon H_m^G(EG;\bfL_{\IZ}^{\langle -\infty \rangle}) \xrightarrow{\cong}  H_m(BG;\bfL) = H_m(N;\bfL).
\]
The $K$-theoretic Farrell-Jones Conjecture applied to the torsionfree group $G$ implies
that $K_n(\IZ G)$ for $n \le -1$, the reduced
projective class group $\widetilde{K}_0(\IZ G)$, and the Whitehead group $\Wh(G)$ vanish.
One concludes from the so called \emph{Rothenberg sequences}, see~\cite[Theorem~17.2 on page~146]{Ranicki(1992a)},
that for $m = n,n+1$ the canonical map
\[
r_m \colon L_m^s(\IZ G) \xrightarrow{\cong} L_m^{\langle -\infty \rangle}(\IZ G)
\]
is bijective. The up to $G$-homotopy unique $G$-map
$i \colon EG = \EGF{G}{\caltr} \to \edub{G}$ induces for all $m \in \IZ$ an isomorphism, 
see Theorem~\ref{the:relative_assembly_maps}~\eqref{the:relative_assembly_maps:L-tor},
\[
H_m^G(i;\bfL_{\IZ}^{\langle -\infty \rangle})  \colon
H_m^G(EG;\bfL_{\IZ}^{\langle -\infty \rangle})  \xrightarrow{\cong}  H_m^G(\edub{G};\bfL_{\IZ}^{\langle -\infty \rangle}).
\]
The $L$-theoretic Farrell-Jones Conjecture predicts the bijectivity of the assembly map
\[
H_n^G(\pr,\bfL_{\IZ}^{\langle -\infty \rangle}) \colon H_n^G(\edub{G};\bfL_{\IZ}^{\langle -\infty \rangle}) 
\xrightarrow{\cong} H_n^G(G/G;\bfL_{\IZ}^{\langle -\infty \rangle}) = L_n^{\langle -\infty \rangle}(\IZ G).
\]
The following diagram commutes
\[
\xymatrix@!C=11em{%
\caln(N) \ar[rr]^-{\sigma_n} \ar[d]_{u_n}^{\cong}
& &
L_n^s(\IZ G)
\\
H_n(N;\bfL\langle 1 \rangle)
\ar[d]_{v_n} & & 
\\
H_n(N;\bfL)
& & 
\\
H_n^G(EG;\bfL_{\IZ}^{\langle -\infty \rangle}) \ar[u]_{w_n}^{\cong} 
\ar[r]^{\cong}_-{H_n^G(i,\bfL_{\IZ}^{\langle -\infty \rangle})}
&
H_n^G(\edub{G};\bfL_{\IZ}^{\langle -\infty \rangle}) 
\ar[r]^{\cong}_-{H_n^G(\pr,\bfL_{\IZ}^{\langle -\infty \rangle})}
&
L_n^{\langle -\infty \rangle}(\IZ G) \ar[uuu]_{r_n}^{\cong}
}
\]
If we replace everywhere $n$ by $n+1$ and in the upper left corner $\caln(N)$ by
$\caln(N\times [0,1],N \times \partial [0,1])$, we get the analogous commutative diagram.
The proof of the commutativity of these diagrams is rather involved and we refer for a
proof for instance to~\cite{Kuehl-Macko-Mole(2013)}.  We conclude from these two diagrams
that $\sigma_n$ is injective and $\sigma_{n+1}$ is bijective since $v_n$ is injective and
$v_{n+1}$ is bijective. Recall that this implies the vanishing of the structure set
$\cals(N)$.  Hence the Farrell-Jones Conjecture~\ref{con:FJC} implies the Borel Conjecture
in dimensions $\ge 5$.

For more information about $L$-groups and surgery theory and the arguments and facts above
we refer for instance 
to~\cite{Cappell-Ranicki-Rosenberg(2000),  Cappell-Ranicki-Rosenberg(2001), Crowley-Lueck-Macko(2019), 
Lueck(2002c), Ranicki(1992),
  Wall(1999)}.


\subsection{The interpretation of the Farrell-Jones assembly map in terms of controlled topology}
\label{subsec:The_interpretation_of_the_Farrell-Jones_assembly_map_in_terms_of_controlled_topology}

We have defined the assembly map appearing in the Farrell-Jones Conjecture as a map
induced by the projection $\edub{G} \to G/G$ for a $G$-homology theory
$ H_*^G(X;\bfE^G)$ or for the  functor $(\bfE^G)_{\%} \colon \GCW \to \Spectra$. We have also
given a homotopy theoretic interpretation in terms of  homotopy colimits and described its universal
property to be the best approximation from the left by an excisive functor.  This
interpretation is good for structural and computational aspects but it turns out that it
is not helpful for the proof that the assembly maps is a weak homotopy equivalence.
There is no direct homotopy theoretic construction of an inverse up to weak homotopy
equivalence known to the author.

For the actual proofs that the assembly maps are weak homotopy equivalences, the
interpretation of the assembly map as a \emph{forget control map}%
\index{assembly map!in terms of forgetting control} is crucial. This
fundamental idea is due to Quinn.

Roughly speaking, one attaches to a metric space certain categories, to these categories
spectra and then takes their homotopy groups, where everything depends on a choice of
certain control conditions which in some sense measure sizes of cycles.  If one requires
certain control conditions, one obtains the source of the assembly map.  If one requires
no control conditions, one obtains the target of the assembly map.  The assembly map
itself is forgetting the control condition.

One of the basic features of a homology theory is excision. It often comes from the fact
that a representing cycle can be arranged to have arbitrarily good control.  An example is the
technique of subdivision which allows to make the representing cycles for singular homology
arbitrarily controlled, i.e., the diameter of the image of a singular simplex
appearing in a singular chain with non-zero coefficient can be arranged to be arbitrarily small.  
This is the key ingredient in the proof that singular homology satisfies excision.
In general one may say that requiring control conditions amounts to implementing homological properties.

With this interpretation it is clear what the main task in the proof of
surjectivity of the assembly map is: \emph{achieve control}, i.e., manipulate
cycles without changing their homology class so that they become sufficiently
controlled.  There is a general principle that a proof of surjectivity also
gives injectivity. Namely, proving injectivity means that one must construct a
cycle whose boundary is a given cycle, i.e., one has to solve a surjectivity
problem in a relative situation. The actual  implementation of this idea
is rather technical. The proof that this forget control version of the assembly map agrees up to weak homotopy
equivalence with the homotopy theoretic one appearing in the Farrell-Jones Conjecture~\ref{con:FJC}
is a direct application of Section~\ref{sec:The_universal_property}. The same is true also for the version of the assembly map
appearing in~\cite[1.6 on page~257]{Farrell-Jones(1993a)}, as explained in~\cite[page~239]{Davis-Lueck(1998)}.

To achieve control one can now use geometric methods. The key ingredients are 
contracting maps and open coverings, transfers, flow spaces and the geometry of the group $G$.

For more information about the general strategy of proofs we refer for instance 
to~\cite{Bartels(2016),Lueck(2010asph),Lueck(2020book)}.


\subsection{The Farrell-Jones Conjecture for Waldhausen's $A$-theory}
\label{subsec:The_Farrell-Jones_Conjecture_for_Waldhausens_A-theory}

Waldhausen has defined for a $CW$-complex $X$ its
algebraic $K$-theory space $A(X)$ in~\cite[Chapter~2]{Waldhausen(1985)}.  
As in the case of algebraic $K$-theory of rings it
will be necessary  to consider a non-connective version. Vogell~\cite{Vogell(1990)}
has defined a delooping of $A(X)$ yielding a non-connective spectrum $\bfA (X)$ for
a $CW$-complex $X$.  This construction actually yields a covariant functor from the
category of topological spaces to the category of spectra. We can assign to a groupoid
$\calg$ its classifying space $B\calg$. Thus we obtain a covariant functor
\begin{equation}
  \bfA \colon \Groupoids \to \Spectra, \quad \calg \mapsto \bfA(B\calg),
\label{bfA_Groupoid_to_Spectra}
\end{equation}
denoted by $\bfA$ again. It respects equivalences, see~\cite[Proposition~2.1.7]{Waldhausen(1985)}  
and~\cite{Vogell(1990)}. If we now take this functor and the family $\calvcyc$ of virtually cyclic subgroups,
we obtain the $A$-theoretic Farrell-Jones Conjecture

\begin{conjecture}[$A$-theoretic Farrell-Jones Conjecture]\label{con:FJC_A-theory}%
\index{Farrell-Jones Conjecture!$A$-theoretic}
  A group $G$ satisfies the \emph{$A$-theoretic Farrell-Jones Conjecture} if 
  the assembly maps induced by the projection
  $\pr \colon \edub{G} \to G/G$
\[
  H_n^G(\pr;\bfA) \colon H^G_n(\edub{G};\bfA) 
\to 
H^G_n(G/G;\bfA)  = \pi_n(\bfA(BG))
\]
is bijective for all $n \in \IZ$.
\end{conjecture}

The $A$-theoretic Farrell-Jones Conjecture~\ref{con:FJC_A-theory} is an
important ingredient in the computation of the group of selfhomeomorphisms of an
aspherical closed manifold in the stable range using the machinery of
Weiss-Williams~\cite{Weiss-Williams(2001)}. Moreover, it is related to Whitehead spaces, 
pseudo-isotopy spaces and spaces of h-cobordisms, see for 
instance~\cite{Dwyer-Weiss-Williams(2003), Waldhausen(1978), Waldhausen(1985),
Waldhausen(1987a), Waldhausen(1987b), Jahren-Rognes-Waldhausen(2013),  Weiss-Williams(2001)}.


\subsection{Relating the assembly maps for $K$-theory and for $A$-theory}
\label{subsec:Relating_the_assembly_maps_for_K-theory_and_for_A-theory}

Let $X$ be a connected $CW$-complex with fundamental group $\pi = \pi_1(X)$.
 Essentially by passing to the cellular $\IZ \pi$-chain complex of the universal covering
one obtains a natural map of (non-connective) spectra, natural in $X$,
called \emph{linearization map}
\begin{eqnarray}
\bfL(X) \colon \bfA(X) & \to & \bfK(\IZ \pi_1(X)).
\label{bfl_bfA_to_bfK}
\end{eqnarray}

The next result follows by combining~\cite[Section~4]{Vogell(1991)} 
and~\cite[Proposition~2.2 and Proposition~2.3]{Waldhausen(1978)}.

\begin{theorem}[Connectivity of the linearization map]%
\label{the:Connectivity_of_the_linearization_map}
\index{assembly map!and the linearization map}
  Let $X$ be a connected $CW$-complex. Then:

\begin{enumerate}

\item\label{the:Connectivity_of_the_linearization_map:2_connected}
The linearization map $\bfL(X)$ of
  $\eqref{bfl_bfA_to_bfK}$ is $2$-connected, i.e., the map
\[
L_n := \pi_n(\bfL(X)) \colon A_n(X) \to K_n(\IZ \pi_1(X))
\]
is bijective for $n \le 1$ and surjective for $n = 2$;

\item\label{the:Connectivity_of_the_linearization_map:rational} 
Rationally the map $L_n$ is bijective for all $n \in \IZ$, provided
that $X$ is aspherical.
\end{enumerate}
\end{theorem}

Thus one obtain a transformation $\bfL \colon \bfA \to \bfK_{\IZ}$ of covariant functors
$\Groupoids \to \Spectra$, where $\bfK_{\IZ}$  and $\bfA$ have been defined in~\eqref{Kalg_GROUPOIDS}
and~\eqref{bfA_Groupoid_to_Spectra}. It induces a commutative diagram
\begin{equation}
\xymatrix@!C=13em{%
H^G_n(\edub{G};\bfA)  \ar[r]^-{H_n^G(\pr;\bfA)} \ar[d]_{H^G_n(\id_{\edub{G}};\bfL)}
&
H^G_n(G/G;\bfA)  = \pi_n(\bfA(BG)) \ar[d]^{\pi_n(\bfL(BG))}
\\
H^G_n(\edub{G};\bfK_{\IZ}) \ar[r]_-{H_n^G(\pr;\bfK_R)}
& 
K_n( \IZ G)
}
\label{diagram_relating_A_to_K}
\end{equation}
where the upper horizontal arrow is the assembly map appearing
in $A$-theoretic Farrell-Jones Conjecture~\ref{con:FJC_A-theory},
the lower  horizontal arrow is the assembly map appearing in the 
$K$-theoretic Farrell-Jones Conjecture~\ref{con:FJC}, 
and both vertical arrows are bijective for $n \le 1$, surjective for $n = 2$ and rationally bijective for
all $n \in \IZ$. In particular the $K$-theoretic Farrell-Jones Conjecture~\ref{con:FJC} for $R = \IZ$
and the $A$-theoretic Farrell-Jones Conjecture~\ref{con:FJC_A-theory} are  rationally equivalent.


\subsection{The status of the Farrell-Jones Conjecture}
\label{subsec:The_status_of_the_Farrell-Jones_Conjecture}

There is a more general version of the Farrell-Jones Conjecture, the so called \emph{Full Farrell-Jones Conjecture},
where one allows coefficients in additive categories and the passage to finite wreath products,
It implies  the Farrell-Jones Conjectures~\ref{con:FJC}. For $A$-theory there is a so called \emph{fiibered} version which
implies Conjecture~\ref{con:FJC_A-theory}.
Let $\calfj$ be the class of groups for which the Full Farrell-Jones Conjecture and the fibered $A$-theoretic
Farrell-Jones Conjecture holds. Notice that then any group in $\calfj$ satisfies in particular
Conjectures~\ref{con:FJC} and~\ref{con:FJC_A-theory}.

\begin{theorem}[The class $\calfj$]%
\index{Farrell-Jones Conjecture!status}
\label{the:status_of_the_Full_Farrell-Jones_Conjecture}\

\begin{enumerate}

\item\label{the:status_of_the_Full_Farrell-Jones_Conjecture:Classes_of_groups}
The following classes of  groups belong to $\calfj$:
\begin{enumerate}

\item\label{the:status_of_the_Full_Farrell-Jones_Conjecture:Classes_of_groups:hyperbolic_groups}
Hyperbolic groups;

\item\label{the:status_of_the_Full_Farrell-Jones_Conjecture:Classes_of_groups:CAT(0)-groups}
Finite dimensional $\CAT(0)$-groups;

\item\label{the:status_of_the_Full_Farrell-Jones_Conjecture:Classes_of_groups:solvable}
Virtually solvable groups;

\item\label{the:status_of_the_Full_Farrell-Jones_Conjecture:Classes_of_groups:lattices}
  (Not necessarily cocompact) lattices in second countable locally compact Hausdorff
  groups with finitely many path components;

\item\label{the:status_of_the_Full_Farrell-Jones_Conjecture:Classes_of_groups:pi_low_dimensional}
  Fundamental groups of (not necessarily compact) connected manifolds (possibly with
  boundary) of dimension $\le 3$;

\item\label{the:status_of_the_Full_Farrell-Jones_Conjecture:Classes_of_groups:GL}
The groups $GL_n(\IQ)$ and $GL_n(F(t))$ for $F(t)$ the function field over a finite field $F$;

\item\label{the:status_of_the_Full_Farrell-Jones_Conjecture:Classes_of_groups:S-arthmetic}
$S$-arithmetic groups;

\item\label{the:status_of_the_Full_Farrell-Jones_Conjecture:Classes_of_groups:mappping_class_groups}
mapping class groups;

\end{enumerate}

\item\label{the:status_of_the_Full_Farrell-Jones_Conjecture:inheritance}
The class $\calfj$ has the following inheritance properties:

\begin{enumerate}

\item \emph{Passing to subgroups}\\
\label{the:status_of_the_Full_Farrell-Jones_Conjecture:inheritance:passing_to_subgroups}
Let $H \subseteq G$ be an inclusion of groups. 
If $G$ belongs to $\calfj$, then $H$ belongs to $\calfj$;

\item \emph{Passing to finite direct products}\\
\label{the:status_of_the_Full_Farrell-Jones_Conjecture:inheritance:Passing_to_finite_direct_products}
If the groups $G_0$ and $G_1$ belong to $\calfj$, then also $G_0 \times G_1$ belongs to $\calfj$;

\item \emph{Group extensions}\\
\label{the:status_of_the_Full_Farrell-Jones_Conjecture:inheritance:group_extensions}
Let $ 1 \to K \to G \to Q \to 1$ be an extension of groups. Suppose that for any cyclic subgroup
$C \subseteq Q$ the group $p^{-1}(C)$ belongs to $\calfj$ 
and that the group $Q$ belongs to $\calfj$.

Then $G$ belongs to $\calfj$;

\item \emph{Directed colimits}\\
\label{the:status_of_the_Full_Farrell-Jones_Conjecture:inheritance:directed_colimits}
Let $\{G_i \mid i \in I\}$ be a direct system of groups indexed by the directed set $I$
(with arbitrary structure maps). Suppose that for each $i \in I$ the group $G_i$ belongs to $\calfj$.

Then the colimit $\colim_{i \in I} G_i$ belongs to $\calfj$;

\item \emph{Passing to finite free products}\\
\label{the:status_of_the_Full_Farrell-Jones_Conjecture:inheritance:Passing_to_finite_free_products}
If  the groups $G_0$ and $G_1$ belong to $\calfj$,  then $G_0 \ast G_1 $ belongs to $\calfj$;

\item \emph{Passing to overgroups of finite index}\\
\label{the:status_of_the_Full_Farrell-Jones_Conjecture:Passing_to_over_groups_of_finite_index}
Let $G$ be an overgroup of $H$ with finite index $[G:H]$.
If $H$ belongs to $\calfj$, then $G$ belongs to $\calfj$;

\end{enumerate}

\end{enumerate}

\end{theorem}
\begin{proof}
See~\cite{Bartels-Bestvina(2016), Bartels-Echterhoff-Lueck(2008colim), Bartels-Farrell-Lueck(2014), Bartels-Lueck(2012annals),
Bartels-Lueck-Reich(2008hyper),  Bartels-Lueck-Reich-Rueping(2014),  Enkelmann-Lueck-Malte-Ullmann-Winges(2018),
Kammeyer-Lueck-Rueping(2016), Kasprowski-Ullmann-Wegner-Winges(2018),
Rueping(2016_S-arithmetic), Wegner(2012),Wegner(2015)}.
\end{proof}

It is not known whether all amenable groups belong to $\calfj$.


\typeout{------------------------- The Baum-Connes Conjecture -------------------------------}

\section{The Baum-Connes Conjecture}
\label{sec:The_Baum-Connes_Conjecture}


\subsection{The Baum-Connes Conjecture}
\label{subsec:The_Baum-Connes_Conjecture}

Recall that the \emph{reduced group $C^*$-algebra $C^*_r(G)$}
is a certain completion of the complex group ring $\IC G$. Namely, there is a canonical embedding of
$\IC G$ into the space $B(l^2(G))$ of bounded operators $L^2(G) \to L^2(G)$ equipped with the supremums norm 
given by the right regular representation, and $C^*_r(G)$ is the norm closure of $\IC G$ in $B(L^2(G))$.
There is a covariant functor respecting equivalences
\begin{equation}
\bfK^{\topo}
\colon \Groupoids \to  \Spectra,
\label{topo_K-theory_reduced}
\end{equation}
such that for every group $G$ and all $n \in \IZ$ we have
\[
\pi_n(\bfK^{\topo}(G))  \cong  K_n^{\topo}(C^*_r(G)),
\]
where $K_n^{\topo}(C^*_r(G))$ is the topological $K$-theory of the reduced group $C^*$-algebra $C^*_r(G)$,
see~\cite{Joachim(2003a)}.
 If we now take this functors and the family $\calfin$ of finite  subgroups,
we obtain

\begin{conjecture}[Baum-Connes Conjecture]\label{con:BCC}
  A group $G$ satisfies the \emph{Baum-Connes  Conjecture}%
\index{Baum-Connes Conjecture}
 if 
  the assembly maps induced by the projection
  $\pr \colon \eub{G} \to G/G$
\[
  H_n^G(\pr;\bfK^{\topo}) \colon H^G_n(\eub{G};\bfK^{\topo}) 
\to 
H^G_n(G/G;\bfK^{\topo})  = K_n^{\topo}(C^*_r(G)).
\]
is bijective for all $n \in \IZ$.
\end{conjecture}

The original version of the Baum-Connes Conjecture is stated in~\cite[Conjecture~3.15 on
page~254]{Baum-Connes-Higson(1994)}.  There is also a version, where the ground field
$\IC$ is replaced by $\IR$.  The complex version of the Baum-Connes
Conjecture~\ref{con:BCC} implies automatically the real version,
see~\cite{Baum-Karoubi(2004),Schick(2004real_versus_complex)}.


\subsection{The Baum-Connes Conjecture with coefficients}
\label{subsec:The_Baum-Connes_Conjecture_with_coefficients}

There is also a more general version of the Baum-Connes
Conjecture~\ref{con:BCC}, where one allows twisted coefficients. However,
there are counterexamples to this more general version, see~\cite[Section
7]{Higson-Lafforgue-Skandalis(2002)}.  There is a new formulation of the Baum-Connes
Conjecture with coefficients in~\cite{Baum-Guentner-Willet(2016expanders)}, where these
counterexamples do not occur anymore. At the time of writing no counterexample to the
Baum-Connes Conjecture~\ref{con:BCC} or to the version 
of~\cite{Baum-Guentner-Willet(2016expanders)}  is known to the author.


\subsection{The interpretation of the Baum Connes assembly map in terms of index  theory}
\label{subsec:The_interpretation_of_the_Baum-Connes_assembly_map_in_terms_of_index_theory}

For applications of the Baum-Connes Conjecture~\ref{con:BCC} it is essential
that the Baum-Connes assembly maps can be interpreted in terms of indices of equivariant
operators with values in $C^*$-algebras. Namely, one assigns to a Kasparov cycle
representing an element in the equivariant $KK$-group $KK^G_n(C_0(X),\IC)$ in the sense of
Kasparov~\cite{Kasparov(1984), Kasparov(1988), Kasparov(1995)} its
$C^*$-valued index in $K_n(C^*_r(G))$ in the sense of
Mishchenko-Fomenko~\cite{Mishchenko-Fomenko(1980)}, thus defining a map
\[
KK_n^G(C_0(X),\IC) \to K_n^{\topo}(C^*_r(G)),
\]
provided that $X$ is proper and cocompact and $C_0(X)$ is the $C^*$-algebra (possibly
without unit) of continuous function $X \to \IC$ vanishing at infinity.  This is the
approach appearing in~\cite{Baum-Connes-Higson(1994)}.\index{assembly map!in terms of index theory}

The other equivalent approach is based on the Kasparov product.  Given a proper cocompact
$G$-$CW$-complex $X$, one can assign to it an element
$[p_X] \in KK^G_0(\IC,C_0(X) \rtimes_r G)$, where $C_0(X) \rtimes_r G$ denotes the
reduced crossed product $C^*$-algebra associated to the $G$-$C^*$-algebra $C_0(X)$. Now
define the Baum-Connes assembly map by the composition of a descent map and a map coming
from the Kasparov product
\begin{multline*}
KK_n^G(C_0(X),\IC)  \xrightarrow{j_r^G} KK_n(C_0(X) \rtimes_r G,C^*_r(G))
\\
\xrightarrow{[p_X] \widehat{\otimes}_{C_0(X) \rtimes_r G} \, -} KK_n(\IC,C^*_r(G)) = K_n^{\topo}(C^*_r(G)).
\end{multline*}
This extends to arbitrary proper $G$-$CW$-complexes $X$ by
defining the source by
\[
K_n^G(C_0(X),\IC) := \colim_{C \subseteq X} K_n^G(C_0(C),\IC),
\]
where $C$ runs through the finite $G$-$CW$-subcomplexes of $Y$ directed
by inclusion. Hence we can take $X = \eub{G}$ above without assuming any
finiteness conditions on $\eub{G}$. For some information about these two
approaches and their identification, at least for torsionfree $G$, we refer
to~\cite{Land(2015)}.

One can identify the original assembly map of~\cite{Baum-Connes-Higson(1994)}
with the assembly map appearing in Conjecture~\ref{con:BCC} using Section~\ref{con:BCC} 
and the fact that  
\[
\colim_{C \subseteq X} H_n^G(C;\bfK^{\topo})  \xrightarrow{\cong} H_n^G(X;\bfK^{\topo})
\]
is an isomorphism. This is explained in~\cite[page~247-248]{Davis-Lueck(1998)},
Unfortunately, the proof is based on an unfinished preprint by
Carlsson-Pedersen-Roe~\cite{Carlsson-Pedersen-Roe(1996)}, where the assembly map appearing
in~\cite[Conjecture~3.15 on page~254]{Baum-Connes-Higson(1994)} is implemented on the
spectrum level.  Another proof of the identification is given
in~\cite[Corollary~8.4]{Hambleton-Pedersen(2004)} and~\cite[Theorem~1.3]{Mitchener(2004)}.


\subsection{Applications of the Baum-Connes Conjecture}
\label{subsec:Applications_of_the_Baum-Connes_Conjecture}

\subsubsection{Computations}
One can carry out \emph{explicite computations} of topological $K$-groups of group $C^*$-algebras and related $C^*$-algebras
by applying methods from algebraic topology to the left side given by a $G$-homology theory and
by finding small models for the classifying spaces of families using the  topology and geometry of groups.
This leads to classification results about certain $C^*$-algebras, see for 
instance~\cite{Davis-Lueck(2013),Echterhoff-Lueck-Phillips-Walters(2010), Langer-Lueck(2012_K-theory),Li-Lueck(2012)}.

\subsubsection{(Modified) Trace Conjecture}
The Baum-Connes Conjecture~\ref{con:BCC} implies the \emph{Trace Conjecture for torsionfree groups}%
\index{Trace Conjecture}
that for a torsionfree group $G$ the image of
\[
\tr_{C^*_r(G)} \colon K_0(C^*_r(G)) \to \IR
\]
consists of the integers. If one drops the condition torsionfree, there is the so called \emph{Modified Trace Conjecture},
which is implied by  Baum-Connes Conjecture~\ref{con:BCC}, see~\cite{Lueck(2002d)}.

\subsubsection{Kadison Conjecture} 
The Baum-Connes Conjecture~\ref{con:BCC} implies the \emph{Kadison Conjecture}%
\index{Kadison Conjecture}
that for a torsionfree group $G$  the only idempotent elements in $C^*_r(G)$ are $0$ and $1$.

\subsubsection{Novikov Conjecture} 
The Baum-Connes Conjecture~\ref{con:BCC} implies the \emph{Novikov Conjecture}.

\subsubsection{The Zero-in-the-spectrum Conjecture} 
The \emph{Zero-in-the-spectrum Conjecture}%
\index{Zero-in-the-spectrum Conjecture}
 says,
if $\widetilde{M}$ is the universal covering of an
aspherical closed Riemannian manifold $M$, then  zero is  in the spectrum of the minimal closure
of the $p$th Laplacian on $\widetilde{M}$  for some $p \in \{  0, 1, \dots , \dim M \}$.
It is a consequence of the Strong Novikov Conjecture and hence of the 
Baum-Connes Conjecture~\ref{con:BCC}, see~\cite[Chapter~12]{Lueck(2002)}.

\subsubsection{The Stable Gromov-Lawson-Rosenberg Conjecture}
Let $\Omega^{\Spin}_n(BG)$ be the bordism group of closed
$\Spin$-manifolds $M$ of dimension n with a reference map to
$BG$. Given an element $[u \colon M \to BG] \in \Omega^{\Spin}_n(BG)$,
we can take the $C^*_r(G;\IR)$-valued index of the equivariant Dirac operator 
associated to the $G$-covering $\overline{M} \to M$ determined by $u$. Thus we get
a homomorphism
\begin{eqnarray} \ind_{C^*_r(G;\IR)} \colon \Omega^{\Spin}_n(BG) & \to & K^{\topo}_n(C^*_r(G;\IR)). 
\label{index_colon_OmegaSPin_n_to_Ktopo_n(Cast_r(G;bbR))}
\end{eqnarray}
A \emph{Bott manifold}%
\index{Bott manifold}
is any simply connected closed $\Spin$-manifold $B$ of dimension $8$ whose
$\widehat{A}$-genus $\widehat{A}(B)$ is $1$. 
We fix such a choice, the particular choice
does not matter for the sequel. Notice that 
$\ind_{C^*_r(\{1\};\IR)}(B) \in K^{\topo}_8(\IR) \cong \IZ$ is a generator
and the product with this element induces the Bott periodicity isomorphisms
$K^{\topo}_n(C^*_r(G;\IR)) \xrightarrow{\cong}  K^{\topo}_{n+8}(C^*_r(G;\IR))$. 
In particular 
\begin{eqnarray}
\ind_{C^*_r(G;\IR)}(M) & = & \ind_{C^*_r(G;\IR)}(M \times B),
\label{ind(M)_is_ind(M_times_B)}
\end{eqnarray}
if we identify $K^{\topo}_n(C^*_r(G;\IR)) = K^{\topo}_{n+8}(C^*_r(G;\IR))$ via Bott periodicity.

\begin{conjecture}[Stable Gromov-Lawson-Rosenberg Conjecture]
\label{con:Stable_Gromov-Lawson-Rosenberg_Conjecture}
\index{Gromov-Lawson-Rosenberg Conjecture}
Let $M$ be a closed connected $\Spin$-manifold of dimension $n \ge 5$. 
Let $u_M \colon M \to B\pi_1(M)$ be the classifying
map of its universal covering.
Then $M \times B^k$ carries for some integer $k \ge
0$ a Riemannian metric with positive scalar curvature if and only if 
\[
\ind_{C^*_r(\pi_1(M);\IR)}([M,u_M])  =  0 \quad  \in K^{\topo}_n(C^*_r(\pi_1(M);\IR)).
\]
\end{conjecture}

If $M$ carries a Riemannian metric with positive scalar curvature, then the index of the
Dirac operator must vanish by the Bochner-Lichnerowicz formula~\cite {Rosenberg(1986b)}.
The converse statement that the vanishing of the index implies the existence of a
Riemannian metric with positive scalar curvature is the hard part of the conjecture. The unstable
version of Conjecture~\ref{con:Stable_Gromov-Lawson-Rosenberg_Conjecture}, where one does
not stabilize with $B^k$, is not true in general, see~\cite{Schick(1998e)}.

 A sketch of the proof of the following result can be found in Stolz~\cite[Section~3]{Stolz(2002)}.  

\begin{theorem}[The Baum-Connes Conjecture implies the Stable Gromov-Lawson-Rosenberg Conjecture]
\label{the:BCC_implies_SGLR}
If the assembly map for the real version of the Baum-Connes Conjecture~\ref{con:BCC}
is injective for the group $G$, then
the Stable Gromov-Lawson-Rosenberg Conjecture~\ref{con:Stable_Gromov-Lawson-Rosenberg_Conjecture} 
is true for all closed $\Spin$-manifolds of dimension $\ge 5$ with $\pi_1(M) \cong G$.
\end{theorem}


\subsubsection{Knot theory} Cochran-Orr-Teichner give in~\cite{Cochran-Orr-Teichner(2003)} new obstructions
for a knot to be slice which are sharper than the Casson-Gordon
invariants. They use $L^2$-signatures and the Baum-Connes Conjecture~\ref{con:BCC}. 
We also refer to the survey article~\cite{Cochran(2004)} about non-commutative geometry and knot theory.


\subsection{The status of the Baum-Connes Conjecture}
\label{subsec:The_status_of_the_Baum-Connes_Conjecture}

Let $\calbc$ be the class of groups for which the
Baum-Connes Conjecture with coefficients, which implies
the Baum-Connes Conjecture~\ref{con:BCC}, is true.
\index{Baum-Connes Conjecture!status}

\begin{theorem}[Status of the  Baum-Connes Conjecture~\ref{con:BCC}]
\label{the:status_of_the_Baum-Connes_Conjecture_with_coefficients}\

\begin{enumerate}

\item\label{the:status_of_the_Baum-Connes_Conjecture_with_coefficients:groups_in_calbc}
The following classes of groups belong to $\calbc$.

\begin{enumerate}

\item\label{the:status_of_the_Baum-Connes_Conjecture_with_coefficients:groups_in_calbc:a-T-menable}
A-T-menable groups;

\item\label{the:status_of_the_Baum-Connes_Conjecture_with_coefficients:groups_in_calbc:hyperbolic}
Hyperbolic groups;

\item\label{the:status_of_the_Baum-Connes_Conjecture_with_coefficients:groups_in_calbc:one-relator_groups}
One-relator groups;

\item\label{the:status_of_the_Baum-Connes_Conjecture_with_coefficients:groups_in_calbc:3-manifolds}
Fundamental groups of  compact $3$-manifolds (possibly with boundary);

\end{enumerate}

\item 
\label{the:status_of_the_Baum-Connes_Conjecture_with_coefficients:inheritance}
The class $\calbc$ has the following inheritance properties:

\begin{enumerate}

\item \emph{Passing to subgroups}\\
\label{the:status_of_the_Baum-Connes_Conjecture_with_coefficients:inheritance:subgroups}
Let $H \subseteq G$ be an inclusion of groups. 
If $G$ belongs to $\calbc$, then $H$ belongs to $\calbc$;

\item \emph{Passing to finite direct products}\\
\label{the:status_of_the_Baum-Connes_Conjecture_with_coefficients:inheritance:direct_products}
If the groups $G_0$ and $G_1$ belong to $\calbc$, the also $G_0 \times G_1$ belongs to $\calbc$;

\item \emph{Group extensions}\\
\label{the:status_of_the_Baum-Connes_Conjecture_with_coefficients:inheritance:group_extensions}
Let $ 1 \to K \to G \to Q \to 1$ be an extension of groups. Suppose that for any finite subgroup
$F \subseteq Q$ the group $p^{-1}(F)$ belongs to $\calbc$ 
and that the group $Q$ belongs to $\calbc$.

Then $G$ belongs to $\calbc$;

\item \emph{Directed unions}\\
\label{the:status_of_the_Baum-Connes_Conjecture_with_coefficients:inheritance:directed_unions}
Let $\{G_i \mid i \in I\}$ be a direct system of subgroups of $G$ indexed by the directed set $I$
such that $G = \bigcup_{i \in I} G_i$. Suppose that $G_i$ belongs to $\calbc$ for every $i \in I$.

Then $G$ belongs to $\calbc$;

\item \emph{Actions on trees}\\
\label{the:status_of_the_Baum-Connes_Conjecture_with_coefficients:inheritance:actions_on_trees}
Let $G$ be a countable discrete group acting without inversion on a tree $T$. 
Then $G$ belongs to $\calbc$
if and only if  the stabilizers of each of the  vertices of $T$ belong to $\calbc$. 

In particular $\calbc$ is closed  under amalgamated products and HNN-extensions. 

\end{enumerate}

\end{enumerate}
\end{theorem}
\begin{proof}See~\cite{Bartels-Echterhoff-Lueck(2008colim),Chabert-Echterhoff(2001b),
Higson-Kasparov(2001), Lafforgue(2012),Mineyev-Yu(2002),Oyono-Oyono(2001), Oyono-Oyono(2001b)}.
\end{proof}

It is not known whether finite-dimensional CAT(0)-groups  and $SL_n(\IZ)$ for $n \ge 3$ belong to $\calbc$.

For more information about the Baum-Connes Conjecture and its applications we refer for instance 
to~\cite{{Baum-Connes-Higson(1994)}, Higson(1998a), Higson-Roe(2005surgeryI),Higson-Roe(2005surgeryII),Higson-Roe(2005surgeryIII),
Lueck(2020book), Lueck-Reich(2005), Mislin-Valette(2003), Piazza-Schick(2007rho), Rosenberg(2016),Schick(2004), Valette(2002)}.


\subsection{Relating the assembly maps of Farrell-Jones to the one of Baum-Connes}
\label{subsec:Ralating_the_assembly_maps_of_Farrell-Jones_to_the_one_of_Baum-Connes}

One can construct 
the following commutative diagram
\index{assembly map!relating the Farrell-Jones assembly map to the Baum-Connes assembly map}

\begin{equation}
\xymatrix{%
H_n^G(\eub{G};\bfL^{\langle - \infty \rangle} _{\IZ}) [1/2] \ar[r]  \ar[d]_l^{\cong}
& 
L^{\langle -\infty \rangle}_n(\IZ G)[1/2] \ar[d]_{\id}^{\cong}
\\
H_n^G(\eub{G};\bfL^{\langle - \infty \rangle} _{\IZ} [1/2]) \ar[r]
& 
L^{\langle -\infty \rangle}_n(\IZ G)[1/2] 
\\
H_n^G(\eub{G};\bfL^p_{\IZ} [1/2]) \ar[d]_{i_1}^{\cong}  \ar[u]^{i_0}_{\cong} \ar[r] 
& 
L^p_n(\IZ G)[1/2] \ar[d]^{j_1}_{\cong} \ar[u]_{j_0}^{\cong}
\\
H_n^G(\eub{G};\bfL^p_{\IQ} [1/2])\ar[d]_{i_2}^{\cong}   \ar[r] 
& 
L^p_n(\IQ G)[1/2] \ar[d]^{j_2}
\\
H_n^G(\eub{G};\bfL^p_{\IR} [1/2]) \ar[d]_{i_3}^{\cong}  \ar[r] 
& 
L^p_n(\IR G)[1/2] \ar[d]^{j_3}
\\
H_n^G(\eub{G};\bfL^p_{C_r^*(?;\IR )} [1/2])  \ar[r] 
&  
L^p_n(C_r^*(G;\IR))[1/2] 
\\
H_n^G(\eub{G};\bfK^{\topo}_{\IR} [1/2])  \ar[r] \ar[u]^{i_4}_{\cong}  
& 
K^{\topo}_n (C_r^*(G;\IR))[1/2] \ar[u]_{j_4}^{\cong}  
\\
H_n^G(\eub{G};\bfK^{\topo}_{\IR})[1/2] \ar@{>->}[d]_{i_5}  \ar[u]_l^{\cong} 
& 
K^{\topo}_n (C_r^*(G;\IR))[1/2] \ar@{>->}[d]^{j_5}  \ar[u]_{\cong}^{\id}
\\
H_n^G(\eub{G};\bfK^{\topo}_{\IC})[1/2] \ar[r] 
&
K_n (C_r^*(G))[1/2] 
}
\label{really_big_diagram}
\end{equation}
where all horizontal maps are assembly maps and the vertical arrows are induced by
transformations of functors $\Groupoids \to \Spectra$. These transformations are induced by change
of rings maps except the one from $\bfK^{\topo}_{\IR}[1/2]$  to $\bfL^p_{C_r^*(?;\IR )} [1/2]$ 
which is much more complicated and carried out in~\cite{Land-Nikolaus(2018)}.
This sophisticated and key ingredient was missing in~\cite[Lemma~22.13~on page~196]{Kreck-Lueck(2005)},
where the existence of such a dagram was claimed. The same remark applies also to~\cite[Theorem~2.7]{Rosenberg(1995)},
see~\cite[Subsection~1.1]{Land-Nikolaus(2018)}.
Actually, it does not exist without inverting two on the spectrum
level. Since it is a weak equivalence, the maps $i_4$ and $j_4$ are bijections.

For any finite group $H$ each of the following maps is known to be a bijection
because of~\cite[Proposition 22.34 on page 252]{Ranicki(1992)} and $\IR H = C_r^*(H;\IR)$
\[
L^p_n(\IZ H)[1/2] \xrightarrow{\cong} L^p_n(\IQ H)[1/2] \xrightarrow{\cong}  L^p_n(\IR H)[1/2]  
\xrightarrow{\cong} L^p_n(C_r^*(H;\IR)).
\]
The natural map $L_n^p(RG)[1/2] \to L_n^{\langle -\infty \rangle}(RG)[1/2]$ is an isomorphism for any $n \in \IZ$,
group $G$ and ring with involution $R$ by the Rothenberg sequence, see~\cite[Theorem~17.2 on page~146]{Ranicki(1992a)}.
Hence we  conclude from the equivariant Atiyah Hirzebruch spectral sequence that the vertical
arrows $i_1$, $i_2$, and $i_3$  are  isomorphisms.  The arrow
$j_1$ is bijective by~\cite[page~376]{Ranicki(1981)}. The maps $l$  are isomorphisms
for general results about localizations.

The lowermost vertical arrows $i_5$ and $j_5$ are known to be split injective because the
inclusion $C_r^*(G;\IR) \to C_r^*(G;\IC)$ induces an isomorphism $C_r^*(G;\IR) \to
C_r^*(G;\IC)^{\IZ/2}$ for the $\IZ/2$-operation coming from complex conjugation $\IC \to
\IC$.  The following conjecture is already raised as a question 
in~\cite[Remark~23.14 on page~197]{Kreck-Lueck(2005)},
see also~\cite[Completion Conjecture in Subsection~5.2]{Land-Nikolaus(2018)}.

\begin{conjecture}[Passage for $L$-theory from  $\IQ G$ to $\IR G$ to $C^*_r(G;\IR)$]
\index{passage for $L$-theory from $\IQ G$ to $\IR G$ to $C^*_r(G;\IR)$}
\label{con:Passage_for_L-theory_from_QG_to_RG_to_Cast_r(G;IR)}
The maps $j_2$ and $j_3$ appearing in diagram~\eqref{really_big_diagram} are bijective.
\end{conjecture}

One easily checks

\begin{lemma}\label{lem:FJC_and_BC} Let $G$ be a group.

  \begin{enumerate}

  \item\label{lem:FJC_and_BC:FJC_and_BC_imply_Passage} Suppose that $G$ satisfies the
    $L$-theoretic Farrell-Jones
    Conjecture~\ref{con:FJC}
    with coefficients in the ring $R$ for $R = \IQ$ and $R = \IR$ and 
    the Baum-Connes Conjecture~\ref{con:BCC}. Then $G$ satisfies
    Conjecture~\ref{con:Passage_for_L-theory_from_QG_to_RG_to_Cast_r(G;IR)};

  \item\label{lem:FJC_and_BC_FJ:equivalent} Suppose that $G$ satisfies
    Conjecture~\ref{con:Passage_for_L-theory_from_QG_to_RG_to_Cast_r(G;IR)}.
    Then $G$ satisfies the $L$-theoretic Farrell-Jones
    Conjecture~\ref{con:FJC}     for the ring $\IZ$ after inverting $2$, if and only if $G$ satisfies the real version
    of the Baum-Connes Conjecture~\ref{con:BCC} after inverting $2$;

  \item\label{lem:FJC_and_BC_FJ:injectivity} Suppose that the assembly map appearing in
    the  Baum-Connes Conjecture~\ref{con:BCC} is (split) injective after
    inverting $2$. Then the assembly map appearing in $L$-theoretic Farrell-Jones
    Conjecture~\ref{con:FJC} with coefficients in the ring for $R = \IZ$ is (split) injective after inverting $2$.

  \end{enumerate}
\end{lemma}


\typeout{------------------------- Topological cyclic homology -------------------------------}

\section{Topological cyclic homology}
\label{sec:Topological_cyclic_homology}

Let $\bfR$ be a (well-pointed connective) symmetric ring spectrum and $p$ be a prime.
 There are  covariant functors respecting equivalences
\begin{eqnarray*}
\bfTHH_{\bfR}
\colon \Groupoids & \to  &\Spectra;
\\
\bfTC_{\bfR,p}
\colon \Groupoids & \to  & \Spectra,
\end{eqnarray*}
such that for every group $G$ and all $n \in \IZ$ we have
\begin{eqnarray*}
\pi_n(\bfTHH_{\bfR}(G))  & \cong & \pi_n(\bfTHH(\bfR[G]));
\\
\pi_n(\bfTC_{\bfR;p}(G))  & \cong & \pi_n(\bfTC(\bfR[G];p)),
\end{eqnarray*}
where $\bfTHH(\bfR[G])$ is the topological Hochschild homology 
and $\bfTHH(\bfR[G];p)$ is the topological cyclic homology  of the group ring spectrum $\bfR[G]$.


\subsection{Topological Hochschild homology}
\label{subsec:Topological_Hochschild_homology}

If we now take the functor $\bfTHH_{\bfR}$ and the family $\calcyc$ of cyclic 
subgroups, we obtain from~\cite[Theorem~1.19]{Lueck-Reich-Rognes-Varisco(2017)}
that the Farrell-Jones Conjecture for topological Hochschild homology is true for all groups.%
\index{Farrell-Jones Conjecture!for topological Hochschild homology}

\begin{theorem}[Topological Hochschild homology] 
\label{con:BCC_maximal}
 The assembly maps induced by the projection
  $\pr \colon \EGF{G}{\calcyc} \to G/G$
\[
  H_n^G(\pr;\bfTHH_{\bfR}) \colon H^G_n(\EGF{G}{\calcyc};\bfTHH_{\bfR}) 
\to 
H^G_n(G/G;\bfTHH_{\bfR})  = \pi_n(\bfTHH(\bfR[G]))
\]
is bijective for all $n \in \IZ$.
\end{theorem}


\subsection{Topological cyclic homology}
\label{subsec:Topological_cyclic_homology}

If we  take the functor $\bfTC_{\bfR;p}$ and the family $\calfin$ of cyclic 
subgroups, we obtain from~\cite[Theorem~1.5]{Lueck-Reich-Rognes-Varisco(2016assembly)}
that the injectivity part of the Farrell-Jones Conjecture for topological cyclic homology 
is true under certain finiteness assumptions
\index{Farrell-Jones Conjecture!for topological cyclic homology}

\begin{theorem}[Split injectivity for topological cyclic homology] 
\label{the:TC_split_injective}
Assume that one for the following conditions hold for the family $\calf$:
\begin{enumerate}
\item\label{the:TC_split_injective:Fin} 
We  have $\calf = \calfin$ and there is a model for $\eub{G}$ of finite type;

\item\label{the:TC_split_injective:calvycy}
We  have $\calf = \calvcyc$ and $G$ is hyperbolic or virtually abelian.

\end{enumerate}

 Then the assembly maps induced by the projection
  $\pr \colon \eub{G} \to G/G$
\[
  H_n^G(\pr;\bfTC_{\bfR;p}) \colon H^G_n(\eub{G};\bfTC_{\bfR;p}) 
\to 
H^G_n(G/G;\bfTHH_{\bfR;p})  = \pi_n(\bfTHH(\bfR[G];p))
\]
is split injective for all $n \in \IZ$.
\end{theorem}

Moreover, we also have, see~\cite[Theorem~1.2]{Lueck-Reich-Rognes-Varisco(2016assembly)}

\begin{theorem}[Topological cyclic homology and finite groups] 
\label{con:TC_finite} 
If $G$ is a finite group,  then the 
assembly map for the family $\calcyc$ of cyclic subgroups 
\[
  H_n^G(\pr;\bfTC_{\bfR;p}) \colon H^G_n(\EGF{G}{\calcyc};\bfTC_{\bfR;p}) 
\to 
H^G_n(G/G;\bfTC_{\bfR;p})  = \pi_n(\bfTC(\bfR[G];p))
\]
is bijective for all $n \in \IZ$.
\end{theorem}

\begin{remark}[The Farrell Jones Conjecture for topological cyclic homology is not true in general]
\label{rem:FJC_for_TC_fails}
There are examples, where the assembly map
\[
H_n^G(\pr;\bfTC_{\bfR;p}) \colon H^G_n(\edub{G};\bfTC_{\bfR;p}) \to \pi_n(\bfTHH(\bfR[G];p))
\]
not surjective,
see~\cite[Theorem~1.6]{Lueck-Reich-Rognes-Varisco(2016assembly)}.  At
least there is a pro-isomorphism for $\bfTC_{\bfR;p}$ with respect to the
family $\calcyc$,
see~\cite[Theorem~1.4]{Lueck-Reich-Rognes-Varisco(2016assembly)}.  The
complications occurring with topological cyclic homology are due to the
fact that smash products and homotopy inverse limit do not commute in
general, see~\cite{Lueck-Reich-Varisco(2003)}.
\end{remark}


\subsection{Relating the assembly maps of Farrell-Jones to the one for topological 
cyclic homology via the cyclotomic trace}
\label{subsec:Relating_the_assembly_maps_of_Farrell-Jones_to_the_one_the_one_for_topological_cyclic_homology}
\index{assembly map!and the cyclotomic trace}
There is an important transformation from algebraic $K$-theory to topological cyclic
homology, the so called \emph{cyclotomic trace}.  It relates the assembly maps for the algebraic
$K$-theory of $\IZ G$ to the cyclic topological homology of the spherical group ring of $G$
and is a key ingredient in proving the rational injectivity of $K$-theoretic assembly maps.
The construction of the cyclotomic trace and the proof of the $K$-theoretic Novikov
conjecture is carried out in the celebrated paper by
Boekstedt-Hsiang-Madsen~\cite{Boekstedt-Hsiang-Madsen(1993)}. The passage from $\caltr$ to
$\calfin$, thus detecting a much larger portion of the algebraic $K$-theory of $\IZ G$ and
proving new results about the Whitehead group $\Wh(G)$, is presented
in~\cite{Lueck-Reich-Rognes-Varisco(2017)}.

For more information about topological cyclic homology we refer for instance 
to~\cite{Boekstedt-Hsiang-Madsen(1993),Dundas-Goodwillie-McCarthy(2013),Nikolaus-Scholze(2017)}.


\typeout{------------------------- The global point of view -------------------------------}

\section{The global point of view}
\label{sec:The_global_point_of_view}

At various occasions it has turned out that one should take a global point of view, i.e.,
one should not consider each group separately, but take into account that in general there
is a theory which can be applied to every group and the values for the various groups are
linked.  This appears for instance in the following definition taken
from~\cite[Section~1]{Lueck(2002b)}.

Let $\alpha\colon H \to G$ be a group homomorphism.  Given an $H$-space $X$, define the
\emph{induction of $X$ with $\alpha$} to be the $G$-space $\ind_{\alpha} X := G
\times_{\alpha} X$, which is the quotient of $G \times X$ by the right $H$-action $(g,x)
\cdot h := (g\alpha(h),h^{-1} x)$ for $h \in H$ and $(g,x) \in G \times X$.

\begin{definition}[Equivariant homology theory]
\label{def:equivariant_homology_theory}
An \emph{equivariant homology theory with values in $\Lambda$-modules}%
\index{equivariant homology theory}
$
\calh^?_n
$
assigns to each group $G$  a $G$-homology theory $\calh^G_*$ 
with values in $\Lambda$-modules 
(in the sense of Definition~\ref{def:G-homology_theory})
together with the following so called \emph{induction structure}:

Given a group homomorphism $\alpha\colon  H \to G$ and  a $H$-$CW$-pair
$(X,A)$, there are for every $n \in \IZ$ natural homomorphisms
\begin{eqnarray}
&\ind_{\alpha} \colon  \calh_n^H(X,A) \to \calh_n^G(\ind_{\alpha}(X,A)) &
\label{induction_structure}
\end{eqnarray}
satisfying:

\begin{itemize}

\item \emph{Compatibility with the boundary homomorphisms}\\[1mm]
$\partial_n^G \circ \ind_{\alpha} = \ind_{\alpha} \circ \partial_n^H$;

\item \emph{Functoriality}\\[1mm]
Let $\beta\colon G \to K$ be another group homomorphism.
Then we have for $n \in \IZ$
\[
\ind_{\beta \circ \alpha}  = \calh^K_n(f_1)\circ\ind_{\beta} \circ \ind_{\alpha} \colon
\calh^H_n(X,A) \to \calh_n^K(\ind_{\beta\circ\alpha}(X,A)),
\]
where $f_1\colon  \ind_{\beta}\ind_{\alpha}(X,A)
\xrightarrow{\cong} \ind_{\beta\circ \alpha}(X,A),
\quad (k,g,x) \mapsto (k\beta(g),x)$
is the natural $K$-homeo\-mor\-phism;

\item \emph{Compatibility with conjugation}\\[1mm]
For $n \in \IZ$, $g \in G$ and a (proper) $G$-$CW$-pair $(X,A)$
the homomorphism $\ind_{c(g)\colon  G \to G}\colon \calh^G_n(X,A)\to
\calh^G_n(\ind_{c(g)\colon G \to G}(X,A))$ agrees with
$\calh_n^G(f_2)$ for the $G$-homeomorphism
$f_2\colon  (X,A) \to \ind_{c(g)\colon G \to G} (X,A)$ which sends
$x$ to $(1,g^{-1}x)$ in $G\times_{c(g)} (X,A)$;

\item \emph{Bijectivity}\\[1mm]
If $\ker(\alpha)$ acts freely on $X\setminus A$, then 
$\ind_{\alpha}\colon  \calh_n^H(X,A) \to \calh_n^G(\ind_{\alpha}(X,A))$
is bijective for all $n \in \IZ$.

\end{itemize}
\end{definition}

Because of the following theorem it will pay off that in Subsection~\ref{subsec:Spectra_over_Groupoids}
we considered  functors defined on $\Groupoids$ and not only on $\OrG$.

\begin{theorem}[Constructing equivariant homology theories using spectra]
\label{the:GROUPOID-spectra_and_equivariant_homology_theories}
Consider a covariant functor $\bfE\colon \Groupoids\to \Spectra$
respecting equivalences.

Then there is  an equivariant homology theory $H_*^?(-;\bfE)$
satisfying 
\[
H^G_n(G/H; \bfE)  \cong  H^H_n(\pt; \bfE)  \cong  \pi_n(\bfE(H))
\]
for every subgroup $H \subseteq G$ of every group $G$ and every $n \in \IZ$.
\end{theorem}
\begin{proof} See~\cite[Proposition~5.6 on page~793]{Lueck-Reich(2005)}.
\end{proof}

The global point of view has been taken up and pursued by Stefan Schwede on the level of spectra in his
book~\cite{Schwede(2018book)}, where global equivariant homotopy theory for compact
Lie groups is developed. To deal with spectra is much more advanced and sophisticated
than with equivariant homology.


\typeout{------------------------- Relative assembly maps -------------------------------}


\section{Relative assembly maps}
\label{sec:Relative_assembly_maps}
In the formulations of the Isomorphism Conjectures above such as the one due to
Farrell-Jones and Baum-Connes it is important to make the family $\calf$ as small as
possible. The largest family we encounter is $\calvcyc$, but there are special cases,
where one can get smaller families. In particular it is desirable to get away with
$\calfin$, since there are often finite models for $\eub{G} =\EGF{G}{\calfin}$, whereas
conjecturally there is a finite model  for $\edub{G} = \EGF{G}{\calvcyc}$ only if $G$ itself
is virtually cyclic, see~\cite[Conjecture~1]{Juan-Pineda-Leary(2006)}.

The general problem is to study and hopefully to prove bijectivity of relative assembly
map associated to two families $\calf \subseteq \calf'$, i.e., of the map induced by the up
to $G$-homotopy unique $G$-map $\EGF{G}{\calf} \to \EGF{G}{\calf'}$
\[
\asmb_{\calf \subseteq \calf'} \colon H_n^G(\EGF{G}{\calf}) \to H_n^G(\EGF{G}{\calg})
\]
\index{assembly map!relative}
for a $G$-homology theory $H_*^G$ with values in $\Lambda$-modules.  
In studying this the global point of view becomes useful.

The main technical result is the so called Transitivity Principle, which we explain next.
For a family $\calf$ of subgroups of $G$ and a subgroup $H \subset G$ we define
a family of subgroups of $H$
\[
\calf|_H = \{ K \cap H \; | \; K \in \calf \}.
\]

\begin{theorem}[Transitivity Principle]\label{the:transitivity}%
\index{Transitivity Principle}
Let $\calh_{\ast}^? ( - )$ be an equivariant homology theory with values in $\Lambda$-modules.
Suppose $\calf \subset \calf'$ are two families of subgroups of $G$.
If for every $H \in \calf^{\prime}$ and every $n \in \IZ$ the assembly map 
\[
\asmb_{\calf|_H \subseteq  \calall} \colon \calh_n^H ( \EGF{H}{\calf|_H} ) \to \calh_n^H ( \pt )
\]
is an isomorphism, then for every $n \in \IZ$ the relative assembly map
\[
\asmb_{\calf \subseteq  \calf^{\prime}} \colon \calh_n^G ( \EGF{G}{\calf} ) \to \calh_n^G ( \EGF{G}{\calf^{\prime}} )
\]
is an isomorphism.
\end{theorem}
\begin{proof} See~\cite[Theorem~65 on page~742]{Lueck-Reich(2005)}.
\end{proof}

One has the following results about diminishing the family of subgroups.  Denote by
$\calvcyc_I$ the family of subgroups of $G$ which are either finite or admit an
epimorphism onto $\IZ$ with finite kernel.  Obviously
$\calfin \subseteq \calvcyc_I \subseteq \calvcyc$.

\begin{theorem}[Relative assembly maps]\label{the:relative_assembly_maps}
\
\begin{enumerate}

\item\label{the:relative_assembly_maps:K-tor_reg}
The relative assembly map for $K$-theory
\[
\asmb_{\caltr \subseteq  \calvcyc} \colon H_n^G ( \EGF{G}{\caltr} ; \bfK_R ) 
\to H_n^G ( \EGF{G}{\calvcyc} ; \bfK_R ) 
\]
is bijective for all $n \in \IZ$, provided that $G$ is torsionfree and $R$ is regular;

\item\label{the:relative_assembly_maps:K-reg}
The relative assembly map for $K$-theory
\[
\asmb_{\calfin \subseteq \calvcyc} \colon H_n^G ( \EGF{G}{\calfin} ; \bfK_R )
\to H_n^G ( \EGF{G}{\calvcyc} ; \bfK_R)
\]
is bijective for all $n \in \IZ$, provided that $R$ is a regular ring containing $\IQ$;

\item\label{the:relative_assembly_maps:K-vcyc_I} 
The relative assembly map for $K$-theory
\[
\asmb_{\calvcyc_I \subseteq \calvcyc} \colon H_n^G(\EGF{G}{\calvcyc_I};\bfK_R) 
\xrightarrow{\cong} H_n(\EGF{G}{\calvcyc};\bfK_R)
\]
is bijective for all $n \in \IZ$;

\item\label{the:relative_assembly_maps:K-fin_to_vycy_rationally} 
The relative assembly map for $K$-theory
\[
\quad \quad \quad \quad \quad \asmb_{\calfin \subseteq \calvcyc} \otimes_{\IZ} \id_{\IQ}
\colon H_n^G ( \EGF{G}{\calfin} ; \bfK_R) \otimes_{\IZ} \IQ \to
H_n^G ( \EGF{G}{\calvcyc} ; \bfK_R) \otimes_{\IZ} \IQ
\]
is bijective for all $n \in \IZ$, provided that $R$ is regular;

\item\label{the:relative_assembly_maps:L-tor} The relative assembly map for $L$-theory
\[
\quad \asmb_{\caltr \subseteq \calvcyc} \colon H_n^G ( \EGF{G}{\caltr} ; \bfL_R^{\langle - \infty \rangle} ) 
\to H_n^G ( \EGF{G}{\calvcyc} ; \bfL_R^{\langle - \infty \rangle}) 
\]
is an isomorphism for all $n \in \IZ$, provided that $G$ is torsionfree;

\item\label{the:relative_assembly_maps:L-vycy_to_vcyc_I} There relative assembly map for $L$-theory
  \[
  \quad \quad \quad \asmb_{\calfin\subseteq \calvcyc_I} \colon 
  H_n^G\bigl(\EGF{G}{\calfin};\bfL_R^{\langle -    \infty\rangle}\bigr) \to
  H_n^G\bigl(\EGF{G}{\calvcyc_I};\bfL_R^{\langle - \infty\rangle}\bigr)
  \] 
  is bijective for all $n \in \IZ$;

\item\label{the:relative_assembly_maps:L-2_inverted}
The relative assembly map  for $L$-theory
\[
\quad \quad \quad  \quad \quad \asmb_{\calfin \subseteq \calvcyc}[1/2] \colon H_n^G ( \EGF{G}{\calfin} ; \bfL_R^{\langle -\infty \rangle} )[1/2]
\to
H_n^G ( \EGF{G}{\calvcyc} ; \bfL_R^{\langle -\infty \rangle} ) [1/2]
\]
is bijective  for all $n \in \IZ$;

\item\label{the:relative_assembly_maps:BC}
The relative assembly map for topological $K$-theory
\[
\quad \asmb_{\calfcyc \subseteq \calfin} \colon H_n^G ( \EGF{G}{\calfcyc} ; \bfK^{\topo} ) \to H_n^G ( \EGF{G}{\calfin} ; \bfK^{\topo} )
\]
is bijective  for all $n \in \IZ$.  This is also true for the real version;

\item\label{the:relative_assembly_maps:fin_vycy_split_injective}
The relative assembly maps for $K$-theory and $L$-theory
\begin{eqnarray*}
\asmb_{\calfin \subseteq  \calvcyc} \colon H_n^G ( \EGF{G}{\calfin} ; \bfK_R )
& \to & 
H_n^G ( \EGF{G}{\calvcyc} ; \bfK_R );
\\
\asmb_{\calfin \subseteq \calvcyc} \colon H_n^G ( \EGF{G}{\calfin} ; \bfL_R^{\langle -\infty \rangle} )
& \to &
H_n^G ( \EGF{G}{\calvcyc} ; \bfL_R^{\langle -\infty \rangle} ),
\end{eqnarray*}
are split injective for all $n \in \IZ$;

\end{enumerate}
\end{theorem}
\begin{proof} 
See~\cite[Theorem~1.3]{Bartels(2003b)},~\cite[Theorem~0.5]{Bartels-Lueck(2007ind)},~\cite{Davis-Quinn-Reich(2011)},~%
\cite[Lemma~4.2]{Lueck(2005heis)},~\cite[Section~2.5]{Lueck-Reich(2005)},%
~\cite[Theorem~0.1 and Theorem~0.3]{Lueck-Steimle(2016splitasmb)},
and~\cite{Matthey-Mislin(2003)}.
\end{proof}

\begin{remark}[Torsionfree groups]
\label{rem:torsionfree_groups}
A typical application is that for a torsionfree group $G$ and a regular ring the
$K$-theoretic Farrell-Jones Conjecture~\ref{con:FJC} implies together with
Theorem~\ref{the:relative_assembly_maps}~\eqref{the:relative_assembly_maps:K-tor_reg} that
the assembly map
\[
H_n(BG;\bfK(R)) = \pi_n(BG_+ \wedge \bfK(R) \to K_n(RG)
\]
is bijective for all $n \in \IZ$. Analogously, for a torsionfree group $G$ the
$L$-theoretic Farrell-Jones Conjecture~\ref{con:FJC} implies together with
Theorem~\ref{the:relative_assembly_maps}~\eqref{the:relative_assembly_maps:L-tor} that
the assembly map
\[
H_n(BG;\bfL^{\-\langle -\infty \rangle}(R)) = \pi_n(BG_+ \wedge \bfL^{\-\langle -\infty \rangle}) \to L^{\-\langle -\infty \rangle}_n(RG)
\]
is bijective for all $n \in \IZ$.
\end{remark}


\typeout{------------------------- Computationally tools -------------------------------}

\section{Computationally tools}
\label{sec:Computationally_tools}

Most computations of $K$- and $L$-groups of group rings are done using the Farrell-Jones
Conjecture~\ref{con:FJC} and the Baum-Connes Conjecture~\ref{con:BCC}. The situation in
the Farrell-Jones Conjecture~\ref{con:FJC} is more complicated than in the Baum-Connes
setting, since the family $\calvcyc$ is much harder to handle than the family $\calfin$.
One can consider $H_n^G(\eub{G};\bfK_R)$ and $H_n^G(\edub{G},\eub{G};\bfK_R)$ separately
because of
Theorem~\ref{the:relative_assembly_maps}~\eqref{the:relative_assembly_maps:fin_vycy_split_injective},
where one considers $\eub{G}$ as a $G$-$CW$-subcomplex of $\edub{G}$.  The term
$H_n^G(\edub{G},\eub{G};\bfK_R)$ involves Nil-terms and UNil-terms, which are hard to
determine. For $H_n^G(\eub{G};\bfK_R)$, $H_n^G(\eub{G};\bfL_R^{\langle -\infty \rangle})$
and $H_n^G(\eub{G};\bfK^{\topo})$ one can use the equivariant Atyiah-Hirzebruch spectral
sequence or the $p$-chain spectral sequence, see Davis-Lueck~\cite{Davis-Lueck(2003)}.
Rationally these groups can often be computed explicitly using equivariant Chern
characters, see~\cite[Section~1]{Lueck(2002b)}. Notice that these can only be constructed
since we take on the global point of view as explained in
Section~\ref{sec:The_global_point_of_view}. Often an important input is 
that one obtains from the geometry of the underlying group nice models for $\eub{G}$ and can
construct $\edub{G}$ from $\eub{G}$ by attaching a tractable family of equivariant cells.

Here are two example, where these ideas lead to a explicite computation, whose outcome is as simple
as one can hope.

\begin{theorem}[Farrell-Jones Conjecture for torsionfree hyperbolic groups for $K$-theory]%
\label{the:FJC_torsionfree_hyperbolic}
  Let $G$ be a torsionfree hyperbolic group.

  \begin{enumerate}

  \item\label{the:FJC_torsionfree_hyperbolic:higher_K-theory}
  We obtain for all $n \in \IZ$ an isomorphism
  \[
  H_n(BG;\bfK (R)) \oplus \bigoplus_C \bigl(\NK_n(R) \oplus
  \NK_n(R)\bigr) \xrightarrow{\cong} K_n(RG),
  \]
  where $C$ runs through a complete system of representatives of the
  conjugacy classes of maximal infinite cyclic subgroups. 
  If $R$ is regular, we have $\NK_n(R) = 0$ for all $n \in \IZ$;

  \item\label{the:FJC_torsionfree_hyperbolic:middle_and_lower_K-theory}
  The  abelian groups $K_n(\IZ G)$ for $n \le -1$, $\widetilde{K}_0(\IZ G)$, and $\Wh(G)$ vanish;

  \item\label{the:FJC_torsionfree_hyperbolic:L-theory}
  We get for every ring $R$ with involution and $n \in \IZ$ an isomorphism
  \[
  H_n(BG;\bfL^{\langle - \infty \rangle}(R))  \xrightarrow{\cong}  L_n^{\langle - \infty \rangle}(RG).
  \]
  For every $j \in \IZ, j \le 2$, and $n \in \IZ$, the natural map
  \[
   L_n^{\langle j  \rangle}(\IZ G) \xrightarrow{\cong}  L_n^{\langle - \infty \rangle}(\IZ G)
  \]
  is bijective;
  
  \item\label{the:FJC_torsionfree_hyperbolic:topological_K-theory}
  We get for every $n \in \IZ$ an isomorphism
 \[
 K_n^{\topo}(BG)  \xrightarrow{\cong}  K_n^{\topo}(C^*_r(G)).
 \]
\end{enumerate}
\end{theorem}
\begin{proof} 
See~\cite[Theorem~1.2]{Lueck-Rosenthal(2014)}.
\end{proof}

\begin{theorem}\label{the:compexified_version}
Suppose that $G$ satisfies the Baum-Connes Conjecture~\ref{con:BCC}
and $K$-theoretic Farrell-Jones 
Conjecture~\ref{con:FJC}  
with coefficients in the ring $\IC$. 
Let $\con(G)_f$ be the set of conjugacy
classes $(g)$ of elements $g \in G$ of finite order. We denote for $g \in G$ by
$C_G\langle g \rangle$ the centralizer in $G$ of the cyclic subgroup generated by $g$.

Then we get the following commutative square, whose horizontal maps are isomorphisms and
complexifications of assembly maps, and whose vertical maps are induced by the obvious
change of theory homomorphisms
\[
\xymatrix{%
\bigoplus_{p+q=n} \bigoplus_{(g) \in \con(G)_f}
H_p(C_G\langle g \rangle;\IC) \otimes_{\IZ}  K_q (\IC)
\ar[r]^-{\cong} \ar[d]
& 
K_n(\IC G) \otimes_{\IZ} \IC
\ar[d]
\\
\bigoplus_{p+q= n} \bigoplus_{(g) \in \con(G)_f}
H_p(C_G\langle g \rangle;\IC) \otimes_{\IZ}  K_q^{\topo}(\IC)
\ar[r]^-{\cong}
&
K_n^{\topo}(C_r^*(G)) \otimes_{\IZ} \IC
}
\]
\end{theorem}
\begin{proof}
See~\cite[Theorem~0.5]{Lueck(2002b)}).
\end{proof}


\typeout{------------ The challenge of extending equivariant homotopy theory to infinite groups ------------}

\section{The challenge of extending equivariant homotopy theory to infinite groups}
\label{sec:The_challenge_of_extending_equivariant_homotopy_to_infinite_groups}

We have seen that it is important to study equivariant homology and homotopy also for
groups which are not necessarily finite.  In particular equivariant KK-theory has
developed into a whole industry, which, however, does not really take the point of view of
spectra instead of $K$-groups and cycles into account. So we encounter

\begin{problem}\label{pro:Equivariant_homotopy} Extend equivariant
  homotopy theory for finite groups to infinite groups, at least in the case of proper
  $G$-actions.
\end{problem}

A few first steps are already in the literature.  We have already explained the notion of
an equivariant homology and the existence of equivariant Chern characters,
see~\cite{Lueck(2002b)}, where the global point of view enters.  There is also a
cohomological version, see~\cite{Lueck(2005c)}.  Topological $K$-theory has systematically
been studied in~\cite{Lueck-Oliver(2001b), Lueck-Oliver(2001a)}, and an attempt of
defining Burnside rings and equivariant cohomotopy for proper $G$-spaces is presented
in~\cite{Lueck(2005r)}.  These do include multiplicative structures. 

An important ingredient in equivariant homotopy theory for finite groups is to stabilize
with unit spheres in finite-dimensional orthogonal representations. However, there are
infinite groups such that any finite-dimensional representation is trivial and therefore
one has to stabilize with equivariant vector bundles,
see~\cite[Remark~6.17]{Lueck(2005r)}. Or one may have to pass even to Hilbert bundles and
equivariant Fredholm operator between these, see~\cite{Phillips(1989)} and 
also~\cite{Lueck-Oliver(2001a)}.

There are various interesting pairings on the group level in the literature, such as
Kasparov products, the action of Swan groups on algebraic K-theory and so on. They all
should be implemented on the spectrum level.  So a systematical study of higher structures for
equivariant spectra over infinite groups has to be carried out and one has to find the
right equivariant homotopy category. This applies also to multiplicative structures and
smash products.  First steps will be presented
in~\cite{Degrijse-Hausmann-Lueck-Patchkoria-Schwede(2019)} using orthogonal spectra.  This
seems to work well for topological $K$-theory, but is probably not adequate for algebraic
$K$-theory. This remark also holds for global equivariant homotopy theory.

A general description of Mackey structures and induction theorems in the sense of Dress is
described in~\cite{Bartels-Lueck(2007ind)}. There are more sophisticated Mackey structure
and transfers in the equivariant homotopy of finite groups, but it is not at all clear
whether and how they extend to infinite groups.

Topological $K$-theory and the Baum-Connes Conjecture make sense and are studied also for
topological groups, e.g., reductive $p$-adic groups and Lie groups. It is conceivable that
also the Farrell-Jones Conjecture has an analogue for Hecke algebras of totally
disconnected groups, see~\cite[Conjecture~119 on page~773]{Lueck-Reich(2005)}. So one can
ask Problem~\ref{pro:Equivariant_homotopy} also for (not necessarily compact) topological
groups instead of infinite (discrete) groups.


\typeout{-------------------------------------- References  ---------------------r------------------}


\begin{thebibliography}{100}

\bibitem{Adams(1974)}
J.~F. Adams.
\newblock {\em Stable homotopy and generalised homology}.
\newblock University of Chicago Press, Chicago, Ill., 1974.
\newblock Chicago Lectures in Mathematics.

\bibitem{Bartels(2016)}
A.~Bartels.
\newblock On proofs of the {F}arrell-{J}ones conjecture.
\newblock In {\em Topology and geometric group theory}, volume 184 of {\em
  Springer Proc. Math. Stat.}, pages 1--31. Springer, [Cham], 2016.

\bibitem{Bartels-Bestvina(2016)}
A.~Bartels and M.~Bestvina.
\newblock The {F}arrell-{J}ones {C}onjecture for mapping class groups.
\newblock Preprint, arXiv:1606.02844 [math.GT], 2016.

\bibitem{Bartels-Echterhoff-Lueck(2008colim)}
A.~Bartels, S.~Echterhoff, and W.~L\"uck.
\newblock Inheritance of isomorphism conjectures under colimits.
\newblock In Cortinaz, Cuntz, Karoubi, Nest, and Weibel, editors, {\em
  {K}-Theory and noncommutative geometry}, EMS-Series of Congress Reports,
  pages 41--70. European Mathematical Society, 2008.

\bibitem{Bartels-Farrell-Lueck(2014)}
A.~Bartels, F.~T. Farrell, and W.~L{\"u}ck.
\newblock The {F}arrell-{J}ones {C}onjecture for cocompact lattices in
  virtually connected {L}ie groups.
\newblock {\em J. Amer. Math. Soc.}, 27(2):339--388, 2014.

\bibitem{Bartels-Lueck(2007ind)}
A.~Bartels and W.~L{\"u}ck.
\newblock Induction theorems and isomorphism conjectures for {$K$}- and
  {$L$}-theory.
\newblock {\em Forum Math.}, 19:379--406, 2007.

\bibitem{Bartels-Lueck(2012annals)}
A.~Bartels and W.~L{\"u}ck.
\newblock The {B}orel conjecture for hyperbolic and {CAT(0)}-groups.
\newblock {\em Ann. of Math. (2)}, 175:631--689, 2012.

\bibitem{Bartels-Lueck-Reich(2008hyper)}
A.~Bartels, W.~L{\"u}ck, and H.~Reich.
\newblock The {$K$}-theoretic {F}arrell-{J}ones conjecture for hyperbolic
  groups.
\newblock {\em Invent. Math.}, 172(1):29--70, 2008.

\bibitem{Bartels-Lueck-Reich(2008appl)}
A.~Bartels, W.~L\"uck, and H.~Reich.
\newblock On the {F}arrell-{J}ones {C}onjecture and its applications.
\newblock {\em Journal of Topology}, 1:57--86, 2008.

\bibitem{Bartels-Lueck-Reich-Rueping(2014)}
A.~{Bartels}, W.~{L\"uck}, H.~{Reich}, and H.~{R\"uping}.
\newblock {K- and L-theory of group rings over $\mathrm{GL}_n(\mathbf Z)$.}
\newblock {\em {Publ. Math., Inst. Hautes \'Etud. Sci.}}, 119:97--125, 2014.

\bibitem{Bartels-Lueck-Weinberger(2010)}
A.~Bartels, W.~L\"uck, and S.~Weinberger.
\newblock On hyperbolic groups with spheres as boundary.
\newblock {\em Journal of Differential Geometry}, 86(1):1--16, 2010.

\bibitem{Bartels(2003b)}
A.~C. Bartels.
\newblock On the domain of the assembly map in algebraic {$K$}-theory.
\newblock {\em Algebr. Geom. Topol.}, 3:1037--1050 (electronic), 2003.

\bibitem{Bass(1976)}
H.~Bass.
\newblock Euler characteristics and characters of discrete groups.
\newblock {\em Invent. Math.}, 35:155--196, 1976.

\bibitem{Baum-Connes-Higson(1994)}
P.~Baum, A.~Connes, and N.~Higson.
\newblock Classifying space for proper actions and ${K}$-theory of group
  ${C}\sp \ast$-algebras.
\newblock In {\em $C\sp \ast$-algebras: 1943--1993 (San Antonio, TX, 1993)},
  pages 240--291. Amer. Math. Soc., Providence, RI, 1994.

\bibitem{Baum-Guentner-Willet(2016expanders)}
P.~Baum, E.~Guentner, and R.~Willett.
\newblock Expanders, exact crossed products, and the {B}aum-{C}onnes
  conjecture.
\newblock {\em Ann. K-Theory}, 1(2):155--208, 2016.

\bibitem{Baum-Karoubi(2004)}
P.~Baum and M.~Karoubi.
\newblock On the {B}aum-{C}onnes conjecture in the real case.
\newblock {\em Q. J. Math.}, 55(3):231--235, 2004.

\bibitem{Boekstedt-Hsiang-Madsen(1993)}
M.~B{\"o}kstedt, W.~C. Hsiang, and I.~Madsen.
\newblock The cyclotomic trace and algebraic ${K}$-theory of spaces.
\newblock {\em Invent. Math.}, 111(3):465--539, 1993.

\bibitem{Cappell-Ranicki-Rosenberg(2000)}
S.~Cappell, A.~Ranicki, and J.~Rosenberg, editors.
\newblock {\em Surveys on surgery theory. {V}ol. 1}.
\newblock Princeton University Press, Princeton, NJ, 2000.
\newblock Papers dedicated to C. T. C. Wall.

\bibitem{Cappell-Ranicki-Rosenberg(2001)}
S.~Cappell, A.~Ranicki, and J.~Rosenberg, editors.
\newblock {\em Surveys on surgery theory. {V}ol. 2}.
\newblock Princeton University Press, Princeton, NJ, 2001.
\newblock Papers dedicated to C. T. C. Wall.

\bibitem{Carlsson-Pedersen-Roe(1996)}
G.~Carlsson and E.~Pedersen.
\newblock Controlled algebra and the baum-connes conjecture.
\newblock in preparation, 1996.

\bibitem{Chabert-Echterhoff(2001b)}
J.~Chabert and S.~Echterhoff.
\newblock Permanence properties of the {B}aum-{C}onnes conjecture.
\newblock {\em Doc. Math.}, 6:127--183 (electronic), 2001.

\bibitem{Cochran(2004)}
T.~D. Cochran.
\newblock Noncommutative knot theory.
\newblock {\em Algebr. Geom. Topol.}, 4:347--398, 2004.

\bibitem{Cochran-Orr-Teichner(2003)}
T.~D. Cochran, K.~E. Orr, and P.~Teichner.
\newblock Knot concordance, {W}hitney towers and {$L\sp 2$}-signatures.
\newblock {\em Ann. of Math. (2)}, 157(2):433--519, 2003.

\bibitem{Crowley-Lueck-Macko(2019)}
D.~Crowley, W.~L\"uck, and T.~Macko.
\newblock {S}urgery {T}heory: {F}oundations.
\newblock book, in preparation, 2019.

\bibitem{Davis-Lueck(1998)}
J.~F. Davis and W.~L{\"u}ck.
\newblock Spaces over a category and assembly maps in isomorphism conjectures
  in ${K}$- and ${L}$-theory.
\newblock {\em $K$-Theory}, 15(3):201--252, 1998.

\bibitem{Davis-Lueck(2003)}
J.~F. Davis and W.~L{\"u}ck.
\newblock The {$p$}-chain spectral sequence.
\newblock {\em $K$-Theory}, 30(1):71--104, 2003.
\newblock Special issue in honor of Hyman Bass on his seventieth birthday. Part
  I.

\bibitem{Davis-Lueck(2013)}
J.~F. Davis and W.~L{\"u}ck.
\newblock The topological ${K}$-theory of certain crystallographic groups.
\newblock {\em Journal of Non-Commutative Geometry}, 7:373--431, 2013.

\bibitem{Davis-Quinn-Reich(2011)}
J.~F. Davis, F.~Quinn, and H.~Reich.
\newblock {Algebraic $K$-theory over the infinite dihedral group: a controlled
  topology approach.}
\newblock {\em J. Topol.}, 4(3):505--528, 2011.

\bibitem{Degrijse-Hausmann-Lueck-Patchkoria-Schwede(2019)}
D.~Degrijse, M.~Hausmann, W.~L\"uck, I.~Patchkoria, and S.~Schwede.
\newblock Proper equivariant stable homotopy theory.
\newblock in preparation, 2019.

\bibitem{Dundas-Goodwillie-McCarthy(2013)}
B.~I. Dundas, T.~G. Goodwillie, and R.~McCarthy.
\newblock {\em The local structure of algebraic {K}-theory}, volume~18 of {\em
  Algebra and Applications}.
\newblock Springer-Verlag London Ltd., London, 2013.

\bibitem{Dwyer-Weiss-Williams(2003)}
W.~Dwyer, M.~Weiss, and B.~Williams.
\newblock A parametrized index theorem for the algebraic {$K$}-theory {E}uler
  class.
\newblock {\em Acta Math.}, 190(1):1--104, 2003.

\bibitem{Echterhoff-Lueck-Phillips-Walters(2010)}
S.~Echterhoff, W.~L{\"u}ck, N.~C. Phillips, and S.~Walters.
\newblock The structure of crossed products of irrational rotation algebras by
  finite subgroups of {${\rm SL}_2(\Bbb Z)$}.
\newblock {\em J. Reine Angew. Math.}, 639:173--221, 2010.

\bibitem{Enkelmann-Lueck-Malte-Ullmann-Winges(2018)}
N.-E. Enkelmann, W.~L\"{u}ck, M.~Pieper, M.~Ullmann, and C.~Winges.
\newblock On the {F}arrell--{J}ones conjecture for {W}aldhausen's {A}--theory.
\newblock {\em Geom. Topol.}, 22(6):3321--3394, 2018.

\bibitem{Farrell(2002)}
F.~T. Farrell.
\newblock The {B}orel conjecture.
\newblock In F.~T. Farrell, L.~G\"ottsche, and W.~L{\"u}ck, editors, {\em High
  dimensional manifold theory}, number~9 in ICTP Lecture Notes, pages 225--298.
  Abdus Salam International Centre for Theoretical Physics, Trieste, 2002.
\newblock Proceedings of the summer school ``High dimensional manifold theory''
  in Trieste May/June 2001, Number~1.
  http://www.ictp.trieste.it/\~{}pub\_off/lectures/vol9.html.

\bibitem{Farrell-Hsiang(1978)}
F.~T. Farrell and W.~C. Hsiang.
\newblock On the rational homotopy groups of the diffeomorphism groups of
  discs, spheres and aspherical manifolds.
\newblock In {\em Algebraic and geometric topology (Proc. Sympos. Pure Math.,
  Stanford Univ., Stanford, Calif., 1976), Part 1}, Proc. Sympos. Pure Math.,
  XXXII, pages 325--337. Amer. Math. Soc., Providence, R.I., 1978.

\bibitem{Farrell-Jones(1990b)}
F.~T. Farrell and L.~E. Jones.
\newblock Rigidity in geometry and topology.
\newblock In {\em Proceedings of the International Congress of Mathematicians,
  Vol.\ I, II (Kyoto, 1990)}, pages 653--663, Tokyo, 1991. Math. Soc. Japan.

\bibitem{Farrell-Jones(1993a)}
F.~T. Farrell and L.~E. Jones.
\newblock Isomorphism conjectures in algebraic ${K}$-theory.
\newblock {\em J. Amer. Math. Soc.}, 6(2):249--297, 1993.

\bibitem{Farrell-Jones-Lueck(2002)}
F.~T. Farrell, L.~E. Jones, and W.~L{\"u}ck.
\newblock A caveat on the isomorphism conjecture in ${L}$-theory.
\newblock {\em Forum Math.}, 14(3):413--418, 2002.

\bibitem{Farrell-Lueck-Steimle(2018)}
T.~Farrell, W.~L\"uck, and W.~Steimle.
\newblock Approximately fibering a manifold over an aspherical one.
\newblock {\em Math. Ann.}, 370(1-2):669--726, 2018.

\bibitem{Ferry-Lueck-Weinberger(2018)}
S.~Ferry, W.~L\"uck, and S.~Weinberger.
\newblock On the stable {C}annon {C}onjecture.
\newblock Preprint, arXiv:1804.00738 [math.GT], 2018.

\bibitem{Ferry-Ranicki-Rosenberg(1995a)}
S.~C. Ferry, A.~A. Ranicki, and J.~Rosenberg, editors.
\newblock {\em Novikov conjectures, index theorems and rigidity. {V}ol. 1}.
\newblock Cambridge University Press, Cambridge, 1995.
\newblock Including papers from the conference held at the Mathematisches
  Forschungsinstitut Oberwolfach, Oberwolfach, September 6--10, 1993.

\bibitem{Ferry-Ranicki-Rosenberg(1995b)}
S.~C. Ferry, A.~A. Ranicki, and J.~Rosenberg, editors.
\newblock {\em Novikov conjectures, index theorems and rigidity. {V}ol. 2}.
\newblock Cambridge University Press, Cambridge, 1995.
\newblock Including papers from the conference held at the Mathematisches
  Forschungsinstitut Oberwolfach, Oberwolfach, September 6--10, 1993.

\bibitem{Hambleton-Pedersen(2004)}
I.~Hambleton and E.~K. Pedersen.
\newblock Identifying assembly maps in {$K$}- and {$L$}-theory.
\newblock {\em Math. Ann.}, 328(1-2):27--57, 2004.

\bibitem{Hebestreit-Land-Lueck-Randel-Williams(2017)}
F.~Hebestreit, W.~L{\"u}ck, M.~Land, and O.~Randal-Williams.
\newblock A vanishing theorem for tautological classes of aspherical manifolds.
\newblock Preprint, arXiv:1705.06232 [math.AT], 2017.

\bibitem{Higson(1998a)}
N.~Higson.
\newblock The {B}aum-{C}onnes conjecture.
\newblock In {\em Proceedings of the International Congress of Mathematicians,
  Vol. II (Berlin, 1998)}, pages 637--646 (electronic), 1998.

\bibitem{Higson-Kasparov(2001)}
N.~Higson and G.~Kasparov.
\newblock ${E}$-theory and ${K}{K}$-theory for groups which act properly and
  isometrically on {H}ilbert space.
\newblock {\em Invent. Math.}, 144(1):23--74, 2001.

\bibitem{Higson-Lafforgue-Skandalis(2002)}
N.~Higson, V.~Lafforgue, and G.~Skandalis.
\newblock Counterexamples to the {B}aum-{C}onnes conjecture.
\newblock {\em Geom. Funct. Anal.}, 12(2):330--354, 2002.

\bibitem{Higson-Roe(2005surgeryI)}
N.~Higson and J.~Roe.
\newblock Mapping surgery to analysis. {I}. {A}nalytic signatures.
\newblock {\em $K$-Theory}, 33(4):277--299, 2005.

\bibitem{Higson-Roe(2005surgeryII)}
N.~Higson and J.~Roe.
\newblock Mapping surgery to analysis. {II}. {G}eometric signatures.
\newblock {\em $K$-Theory}, 33(4):301--324, 2005.

\bibitem{Higson-Roe(2005surgeryIII)}
N.~Higson and J.~Roe.
\newblock Mapping surgery to analysis. {III}. {E}xact sequences.
\newblock {\em $K$-Theory}, 33(4):325--346, 2005.

\bibitem{Joachim(2003a)}
M.~Joachim.
\newblock {$K$}-homology of {$C\sp \ast$}-categories and symmetric spectra
  representing {$K$}-homology.
\newblock {\em Math. Ann.}, 327(4):641--670, 2003.

\bibitem{Juan-Pineda-Leary(2006)}
D.~Juan-Pineda and I.~J. Leary.
\newblock On classifying spaces for the family of virtually cyclic subgroups.
\newblock In {\em Recent developments in algebraic topology}, volume 407 of
  {\em Contemp. Math.}, pages 135--145. Amer. Math. Soc., Providence, RI, 2006.

\bibitem{Kammeyer-Lueck-Rueping(2016)}
H.~Kammeyer, W.~L{\"u}ck, and H.~R{\"u}ping.
\newblock The {F}arrell--{J}ones conjecture for arbitrary lattices in virtually
  connected {L}ie groups.
\newblock {\em Geom. Topol.}, 20(3):1275--1287, 2016.

\bibitem{Kasparov(1984)}
G.~G. Kasparov.
\newblock Operator {$K$}-theory and its applications: elliptic operators, group
  representations, higher signatures, {$C^\ast$}-extensions.
\newblock In {\em Proceedings of the {I}nternational {C}ongress of
  {M}athematicians, {V}ol.\ 1, 2 ({W}arsaw, 1983)}, pages 987--1000, Warsaw,
  1984. PWN.

\bibitem{Kasparov(1988)}
G.~G. Kasparov.
\newblock Equivariant ${K}{K}$-theory and the {N}ovikov conjecture.
\newblock {\em Invent. Math.}, 91(1):147--201, 1988.

\bibitem{Kasparov(1995)}
G.~G. Kasparov.
\newblock ${K}$-theory, group ${C}\sp *$-algebras, and higher signatures
  (conspectus).
\newblock In {\em Novikov conjectures, index theorems and rigidity, Vol.\ 1
  (Oberwolfach, 1993)}, pages 101--146. Cambridge Univ. Press, Cambridge, 1995.

\bibitem{Kasprowski-Ullmann-Wegner-Winges(2018)}
D.~Kasprowski, M.~Ullmann, C.~Wegner, and C.~Winges.
\newblock The {$A$}-theoretic {F}arrell-{J}ones conjecture for virtually
  solvable groups.
\newblock {\em Bull. Lond. Math. Soc.}, 50(2):219--228, 2018.

\bibitem{Kreck-Lueck(2005)}
M.~Kreck and W.~L{\"u}ck.
\newblock {\em The {N}ovikov conjecture: Geometry and algebra}, volume~33 of
  {\em Oberwolfach Seminars}.
\newblock Birkh\"auser Verlag, Basel, 2005.

\bibitem{Kuehl-Macko-Mole(2013)}
P.~K{\"u}hl, T.~Macko, and A.~Mole.
\newblock The total surgery obstruction revisited.
\newblock {\em M\"unster J. Math.}, 6:181--269, 2013.

\bibitem{Lafforgue(2012)}
V.~Lafforgue.
\newblock {The Baum-Connes conjecture with coefficients for hyperbolic groups.
  (La conjecture de baum-connes \`a coefficients pour les groupes
  hyperboliques.)}.
\newblock {\em J. Noncommut. Geom.}, 6(1):1--197, 2012.

\bibitem{Land(2015)}
M.~Land.
\newblock The analytical assembly map and index theory.
\newblock {\em J. Noncommut. Geom.}, 9(2):603--619, 2015.

\bibitem{Land-Nikolaus(2018)}
M.~Land and T.~Nikolaus.
\newblock On the relation between {$K$}- and {$L$}-theory of {$C^*$}-algebras.
\newblock {\em Math. Ann.}, 371(1-2):517--563, 2018.

\bibitem{Langer-Lueck(2012_K-theory)}
M.~Langer and W.~L{\"u}ck.
\newblock Topological {$K$}-theory of the group {$C^*$}-algebra of a
  semi-direct product {$\Bbb Z^n\rtimes\Bbb Z/m$} for a free conjugation
  action.
\newblock {\em J. Topol. Anal.}, 4(2):121--172, 2012.

\bibitem{Li-Lueck(2012)}
X.~Li and W.~L\"uck.
\newblock {$K$}-theory for ring {$C^*$}-algebras -- the case of number fields
  with higher roots of unity.
\newblock {\em Journal of Topology and Analysis 4 (4)}, pages 449--479, 2012.

\bibitem{Lueck(1989)}
W.~L{\"u}ck.
\newblock {\em Transformation groups and algebraic ${K}$-theory}, volume 1408
  of {\em Lecture Notes in Mathematics}.
\newblock Springer-Verlag, Berlin, 1989.

\bibitem{Lueck(2002c)}
W.~L{\"u}ck.
\newblock A basic introduction to surgery theory.
\newblock In F.~T. Farrell, L.~G\"ottsche, and W.~L{\"u}ck, editors, {\em High
  dimensional manifold theory}, number~9 in ICTP Lecture Notes, pages 1--224.
  Abdus Salam International Centre for Theoretical Physics, Trieste, 2002.
\newblock Proceedings of the summer school ``High dimensional manifold theory''
  in Trieste May/June 2001, Number~1.
  http://www.ictp.trieste.it/\~{}pub\_off/lectures/vol9.html.

\bibitem{Lueck(2002b)}
W.~L{\"u}ck.
\newblock Chern characters for proper equivariant homology theories and
  applications to ${K}$- and ${L}$-theory.
\newblock {\em J. Reine Angew. Math.}, 543:193--234, 2002.

\bibitem{Lueck(2002)}
W.~L{\"u}ck.
\newblock {\em {$L\sp 2$}-{I}nvariants: {T}heory and {A}pplications to
  {G}eometry and \mbox{{$K$}-{T}heory}}, volume~44 of {\em Ergebnisse der
  Mathematik und ihrer Grenzgebiete. 3.~Folge. A Series of Modern Surveys in
  Mathematics [Results in Mathematics and Related Areas. 3rd Series. A Series
  of Modern Surveys in Mathematics]}.
\newblock Springer-Verlag, Berlin, 2002.

\bibitem{Lueck(2002d)}
W.~L{\"u}ck.
\newblock The relation between the {B}aum-{C}onnes conjecture and the trace
  conjecture.
\newblock {\em Invent. Math.}, 149(1):123--152, 2002.

\bibitem{Lueck(2005r)}
W.~L{\"u}ck.
\newblock The {B}urnside ring and equivariant stable cohomotopy for infinite
  groups.
\newblock {\em Pure Appl. Math. Q.}, 1(3):479--541, 2005.

\bibitem{Lueck(2005c)}
W.~L{\"u}ck.
\newblock Equivariant cohomological {C}hern characters.
\newblock {\em Internat. J. Algebra Comput.}, 15(5-6):1025--1052, 2005.

\bibitem{Lueck(2005heis)}
W.~L{\"u}ck.
\newblock {$K$}- and {$L$}-theory of the semi-direct product of the discrete
  3-dimensional {H}eisenberg group by {${\Bbb Z}/4$}.
\newblock {\em Geom. Topol.}, 9:1639--1676 (electronic), 2005.

\bibitem{Lueck(2005s)}
W.~L{\"u}ck.
\newblock Survey on classifying spaces for families of subgroups.
\newblock In {\em Infinite groups: geometric, combinatorial and dynamical
  aspects}, volume 248 of {\em Progr. Math.}, pages 269--322. Birkh\"auser,
  Basel, 2005.

\bibitem{Lueck(2010asph)}
W.~L{\"u}ck.
\newblock Survey on aspherical manifolds.
\newblock In A.~Ran, H.~te~Riele, and J.~Wiegerinck, editors, {\em Proceedings
  of the 5-th European Congress of Mathematics Amsterdam 14 -18 July 2008},
  pages 53--82. EMS, 2010.

\bibitem{Lueck(2020book)}
W.~L\"uck.
\newblock {I}somorphism {C}onjectures in {$K$}- and {$L$}-theory.
\newblock in preparation, 2020.

\bibitem{Lueck-Oliver(2001b)}
W.~L{\"u}ck and B.~Oliver.
\newblock Chern characters for the equivariant ${K}$-theory of proper
  ${G}$-{C}{W}-complexes.
\newblock In {\em Cohomological methods in homotopy theory (Bellaterra, 1998)},
  pages 217--247. Birkh\"auser, Basel, 2001.

\bibitem{Lueck-Oliver(2001a)}
W.~L{\"u}ck and B.~Oliver.
\newblock The completion theorem in ${K}$-theory for proper actions of a
  discrete group.
\newblock {\em Topology}, 40(3):585--616, 2001.

\bibitem{Lueck-Reich(2005)}
W.~L{\"u}ck and H.~Reich.
\newblock The {B}aum-{C}onnes and the {F}arrell-{J}ones conjectures in {$K$}-
  and {$L$}-theory.
\newblock In {\em Handbook of $K$-theory. Vol. 1, 2}, pages 703--842. Springer,
  Berlin, 2005.

\bibitem{Lueck-Reich-Rognes-Varisco(2016assembly)}
W.~L{\"u}ck, H.~{R}eich, J.~{R}ognes, and M.~{V}arisco.
\newblock Assembly maps for topological cyclic homology of group algebras.
\newblock Preprint, arXiv:1607.03557 [math.KT], to appear in Crelle, 2016.

\bibitem{Lueck-Reich-Rognes-Varisco(2017)}
W.~L{\"u}ck, H.~Reich, J.~Rognes, and M.~Varisco.
\newblock Algebraic {K}-theory of group rings and the cyclotomic trace map.
\newblock {\em Adv. Math.}, 304:930--1020, 2017.

\bibitem{Lueck-Reich-Varisco(2003)}
W.~L{\"u}ck, H.~Reich, and M.~Varisco.
\newblock Commuting homotopy limits and smash products.
\newblock {\em $K$-Theory}, 30(2):137--165, 2003.
\newblock Special issue in honor of Hyman Bass on his seventieth birthday. Part
  II.

\bibitem{Lueck-Rosenthal(2014)}
W.~{L\"uck} and D.~{Rosenthal}.
\newblock {On the $K$- and $L$-theory of hyperbolic and virtually finitely
  generated abelian groups.}
\newblock {\em {Forum Math.}}, 26(5):1565--1609, 2014.

\bibitem{Lueck-Steimle(2016splitasmb)}
W.~L{\"u}ck and W.~Steimle.
\newblock Splitting the relative assembly map, {N}il-terms and involutions.
\newblock {\em Ann. K-Theory}, 1(4):339--377, 2016.

\bibitem{Matthey-Mislin(2003)}
M.~Matthey and G.~Mislin.
\newblock Equivariant ${K}$-homology and restriction to finite cyclic
  subgroups.
\newblock Preprint, 2003.

\bibitem{Mineyev-Yu(2002)}
I.~Mineyev and G.~Yu.
\newblock The {B}aum-{C}onnes conjecture for hyperbolic groups.
\newblock {\em Invent. Math.}, 149(1):97--122, 2002.

\bibitem{Mishchenko-Fomenko(1980)}
A.~S. Mi{\v{s}}{\v{c}}enko and A.~T. Fomenko.
\newblock The index of elliptic operators over ${C}\sp{\ast} $-algebras.
\newblock {\em Mathematics of the USSR-Izvestiya}, 15(1):87--112, 1980.

\bibitem{Mislin-Valette(2003)}
G.~Mislin and A.~Valette.
\newblock {\em Proper group actions and the {B}aum-{C}onnes conjecture}.
\newblock Advanced Courses in Mathematics. CRM Barcelona. Birkh\"auser Verlag,
  Basel, 2003.

\bibitem{Mitchener(2004)}
P.~D. Mitchener.
\newblock {$C^*$}-categories, groupoid actions, equivariant {$KK$}-theory, and
  the {B}aum-{C}onnes conjecture.
\newblock {\em J. Funct. Anal.}, 214(1):1--39, 2004.

\bibitem{Nikolaus-Scholze(2017)}
T.~Nikolaus and P.~Scholze.
\newblock On topological cyclic homology.
\newblock Preprint, arXiv:1707.01799 [math.AT], 2017.

\bibitem{Oyono-Oyono(2001)}
H.~Oyono-Oyono.
\newblock {B}aum-{C}onnes {C}onjecture and extensions.
\newblock {\em J. Reine Angew. Math.}, 532:133--149, 2001.

\bibitem{Oyono-Oyono(2001b)}
H.~Oyono-Oyono.
\newblock Baum-{C}onnes conjecture and group actions on trees.
\newblock {\em $K$-Theory}, 24(2):115--134, 2001.

\bibitem{Phillips(1989)}
N.~C. Phillips.
\newblock {\em Equivariant ${K}$-theory for proper actions}.
\newblock Longman Scientific \& Technical, Harlow, 1989.

\bibitem{Piazza-Schick(2007rho)}
P.~Piazza and T.~Schick.
\newblock Bordism, rho-invariants and the {B}aum-{C}onnes conjecture.
\newblock {\em J. Noncommut. Geom.}, 1(1):27--111, 2007.

\bibitem{Quinn(1971)}
F.~Quinn.
\newblock {$^{B}({\rm TOP}_{n})^{\ast \ast \ast bt\ast \ast }$} and the surgery
  obstruction.
\newblock {\em Bull. Amer. Math. Soc.}, 77:596--600, 1971.

\bibitem{Quinn(1995a)}
F.~Quinn.
\newblock Assembly maps in bordism-type theories.
\newblock In {\em Novikov conjectures, index theorems and rigidity, Vol.\ 1
  (Oberwolfach, 1993)}, pages 201--271. Cambridge Univ. Press, Cambridge, 1995.

\bibitem{Ranicki(1981)}
A.~A. Ranicki.
\newblock {\em Exact sequences in the algebraic theory of surgery}.
\newblock Princeton University Press, Princeton, N.J., 1981.

\bibitem{Ranicki(1992)}
A.~A. Ranicki.
\newblock {\em Algebraic ${L}$-theory and topological manifolds}.
\newblock Cambridge University Press, Cambridge, 1992.

\bibitem{Ranicki(1992a)}
A.~A. Ranicki.
\newblock {\em Lower ${K}$- and ${L}$-theory}.
\newblock Cambridge University Press, Cambridge, 1992.

\bibitem{Rosenberg(1986b)}
J.~Rosenberg.
\newblock ${C}\sp \ast$-algebras, positive scalar curvature and the {N}ovikov
  conjecture. {I}{I}{I}.
\newblock {\em Topology}, 25:319--336, 1986.

\bibitem{Rosenberg(1995)}
J.~Rosenberg.
\newblock Analytic {N}ovikov for topologists.
\newblock In {\em Novikov conjectures, index theorems and rigidity, Vol.\ 1
  (Oberwolfach, 1993)}, pages 338--372. Cambridge Univ. Press, Cambridge, 1995.

\bibitem{Rosenberg(2016)}
J.~Rosenberg.
\newblock Structure and applications of real {$C^*$}-algebras.
\newblock In {\em Operator algebras and their applications}, volume 671 of {\em
  Contemp. Math.}, pages 235--258. Amer. Math. Soc., Providence, RI, 2016.

\bibitem{Rueping(2016_S-arithmetic)}
H.~R{\"u}ping.
\newblock The {F}arrell--{J}ones conjecture for {$S$}-arithmetic groups.
\newblock {\em J. Topol.}, 9(1):51--90, 2016.

\bibitem{Schick(1998e)}
T.~Schick.
\newblock A counterexample to the (unstable) {G}romov-{L}awson-{R}osenberg
  conjecture.
\newblock {\em Topology}, 37(6):1165--1168, 1998.

\bibitem{Schick(2004)}
T.~Schick.
\newblock Index theory and the {B}aum-{C}onnes conjecture.
\newblock In {\em Geometry {S}eminars. 2001-2004 ({I}talian)}, pages 231--280.
  Univ. Stud. Bologna, Bologna, 2004.

\bibitem{Schick(2004real_versus_complex)}
T.~Schick.
\newblock Real versus complex {$K$}-theory using {K}asparov's bivariant
  {$KK$}-theory.
\newblock {\em Algebr. Geom. Topol.}, 4:333--346, 2004.

\bibitem{Schwede(2018book)}
S.~Schwede.
\newblock {\em Global homotopy theory}, volume~34 of {\em New Mathematical
  Monographs}.
\newblock Cambridge University Press, Cambridge, 2018.

\bibitem{Stolz(2002)}
S.~Stolz.
\newblock Manifolds of positive scalar curvature.
\newblock In T.~Farrell, L.~G\"ottsche, and W.~L{\"u}ck, editors, {\em High
  dimensional manifold theory}, number~9 in ICTP Lecture Notes, pages 661--708.
  Abdus Salam International Centre for Theoretical Physics, Trieste, 2002.
\newblock Proceedings of the summer school ``High dimensional manifold theory''
  in Trieste May/June 2001, Number~2.
  http://www.ictp.trieste.it/\~{}pub\_off/lectures/vol9.html.

\bibitem{Dieck(1972)}
T.~tom Dieck.
\newblock Orbittypen und \"aquivariante {H}omologie. {I}.
\newblock {\em Arch. Math. (Basel)}, 23:307--317, 1972.

\bibitem{Valette(2002)}
A.~Valette.
\newblock {\em Introduction to the {B}aum-{C}onnes conjecture}.
\newblock Birkh\"auser Verlag, Basel, 2002.
\newblock From notes taken by Indira Chatterji, With an appendix by Guido
  Mislin.

\bibitem{Vogell(1990)}
W.~Vogell.
\newblock Algebraic {$K$}-theory of spaces, with bounded control.
\newblock {\em Acta Math.}, 165(3-4):161--187, 1990.

\bibitem{Vogell(1991)}
W.~Vogell.
\newblock Boundedly controlled algebraic {$K$}-theory of spaces and its linear
  counterparts.
\newblock {\em J. Pure Appl. Algebra}, 76(2):193--224, 1991.

\bibitem{Waldhausen(1978)}
F.~Waldhausen.
\newblock Algebraic {$K$}-theory of topological spaces. {I}.
\newblock In {\em Algebraic and geometric topology (Proc. Sympos. Pure Math.,
  Stanford Univ., Stanford, Calif., 1976), Part 1}, Proc. Sympos. Pure Math.,
  XXXII, pages 35--60. Amer. Math. Soc., Providence, R.I., 1978.

\bibitem{Waldhausen(1985)}
F.~Waldhausen.
\newblock Algebraic ${K}$-theory of spaces.
\newblock In {\em Algebraic and geometric topology (New Brunswick, N.J.,
  1983)}, pages 318--419. Springer-Verlag, Berlin, 1985.

\bibitem{Waldhausen(1987a)}
F.~Waldhausen.
\newblock Algebraic ${K}$-theory of spaces, concordance, and stable homotopy
  theory.
\newblock In {\em Algebraic topology and algebraic $K$-theory (Princeton, N.J.,
  1983)}, pages 392--417. Princeton Univ. Press, Princeton, NJ, 1987.

\bibitem{Waldhausen(1987b)}
F.~Waldhausen.
\newblock An outline of how manifolds relate to algebraic {$K$}-theory.
\newblock In {\em Homotopy theory (Durham, 1985)}, volume 117 of {\em London
  Math. Soc. Lecture Note Ser.}, pages 239--247. Cambridge Univ. Press,
  Cambridge, 1987.

\bibitem{Jahren-Rognes-Waldhausen(2013)}
F.~Waldhausen, B.~Jahren, and J.~Rognes.
\newblock {\em Spaces of {PL} manifolds and categories of simple maps}, volume
  186 of {\em Annals of Mathematics Studies}.
\newblock Princeton University Press, Princeton, NJ, 2013.

\bibitem{Wall(1999)}
C.~T.~C. Wall.
\newblock {\em Surgery on compact manifolds}, volume~69 of {\em Mathematical
  Surveys and Monographs}.
\newblock American Mathematical Society, Providence, RI, second edition, 1999.
\newblock Edited and with a foreword by A. A. Ranicki.

\bibitem{Wegner(2012)}
C.~Wegner.
\newblock The {$K$}-theoretic {F}arrell-{J}ones conjecture for {CAT}(0)-groups.
\newblock {\em Proc. Amer. Math. Soc.}, 140(3):779--793, 2012.

\bibitem{Wegner(2015)}
C.~Wegner.
\newblock The {F}arrell-{J}ones conjecture for virtually solvable groups.
\newblock {\em J. Topol.}, 8(4):975--1016, 2015.

\bibitem{Weiss-Williams(1995a)}
M.~Weiss and B.~Williams.
\newblock Assembly.
\newblock In {\em Novikov conjectures, index theorems and rigidity, Vol.\ 2
  (Oberwolfach, 1993)}, pages 332--352. Cambridge Univ. Press, Cambridge, 1995.

\bibitem{Weiss-Williams(2001)}
M.~Weiss and B.~Williams.
\newblock Automorphisms of manifolds.
\newblock In {\em Surveys on surgery theory, Vol. 2}, volume 149 of {\em Ann.
  of Math. Stud.}, pages 165--220. Princeton Univ. Press, Princeton, NJ, 2001.

\end{thebibliography}





\end{document}